\documentclass[11pt,a4paper]{amsart}

\usepackage[left=2.8cm, top=3cm,bottom=3cm,right=2.8cm]{geometry}

\usepackage{mathtools,amssymb,amsthm,mathrsfs,calc,graphicx,stmaryrd,xcolor,dsfont,pgfplots,bbm,float,array,epsfig,color}
\usepackage[numbers,square]{natbib}
\usepackage[british]{babel}
\usepackage{amsfonts,amsmath,amsthm,amssymb}
\usepackage{verbatim}
\usepackage{multirow}
\usepackage{tikz}
\usepackage{bbm}
\usepackage{enumitem}
\usepackage{subcaption}

\numberwithin{equation}{section}
\numberwithin{figure}{section}

\newtheorem{thm}{Theorem}[section]
\newtheorem{lem}[thm]{Lemma}

\newtheorem{prop}[thm]{Proposition}

\theoremstyle{remark}
\newtheorem{rem}[thm]{Remark}

\theoremstyle{definition}

\numberwithin{equation}{section}

\def\P{{\rm P}}
\def\Z{{\mathbb{Z}}}

\def\R{{\mathbb{R}}}
\def\V{{\mathbb{V}}}
\def\F{{\mathcal{F}}}

\newcommand{\od}{\overset{{\rm d}}{=}}
\newcommand{\tod}{\overset{{\rm d}}{\longrightarrow}}
\newcommand{\topr}{\overset{\mathbb{P}}{\longrightarrow}}

\newcommand{\PP}{\mathbb{P}}
\newcommand{\E}{{\rm E}}
\newcommand{\EE}{\mathbb{E}}
\newcommand{\Y}{\mathbb{Y}}
\newcommand{\G}{\mathcal{G}}
\newcommand{\X}{\mathcal{X}}

\newcommand{\Po}{{\rm P}_{\omega}}
\newcommand{\Eo}{{\rm E}_{\omega}}

\newcommand{\N}{\mathbb{N}}

\newcommand{\var}{{\rm Var}}
\newcommand{\varo}{\var_\omega}

\newcommand{\Rb}{{\bar{R}}}
\newcommand{\U}{\mathbb{U}}

\newcommand{\eps}{\varepsilon}
\newcommand{\wt}{\widetilde}
\DeclareMathOperator{\1}{\mathbbm{1}}

\begin{document}

\title[Favourite sites of a~random walk in a~sparse random environment]{On favourite sites of a~random walk in moderately sparse random environment}
\author[A. Ko{\l}odziejska]{Alicja Ko{\l}odziejska}
\thanks{Mathematical
		Institute, University of Wroc{\l}aw, Pl.~Grunwaldzki 2, 50-384 Wroc{\l}aw, Poland. E-mail:
		{\tt alicja.kolodziejska@math.uni.wroc.pl}}
\begin{abstract}
	We study the favourite sites of a~random walk evolving in a~sparse random environment on the set of integers. The walker moves symmetrically apart from some randomly chosen sites where we impose random drift. We prove annealed limit theorems for the time the walk spends in its favourite sites in two cases. The first one, in which it is the distribution of the drift that determines the limiting behaviour of the walk, is a~generalization of known results for a~random walk in i.i.d.\ random environment. In the second case a~new behaviour appears, caused by the sparsity of the environment.
\end{abstract}
\date{}
\maketitle
{\footnotesize \noindent \textbf{Keywords:} {random walk in random environment, branching process in random environment, sparse random environment, local times.} \\
\textbf{MSC2020 subject classifications:} primary: 60K37; secondary: 60F05.}
\thispagestyle{empty}

\section{Introduction}\label{sec:intro}

One of the most classic and well studied stochastic processes is a~simple symmetric random walk on the set of integers, which models the~movement of a~single particle in one-dimensional, homogeneous medium. The simplicity of the model allows to analyse it with the help of such classic results as the strong law of large numbers or the central limit theorem; however, its homogeneity is not always desired. In many applications one would like to consider some obstacles or impurities of the medium, possibly placed randomly, that would have impact on the movement of the particle. One of the ways of defining such random environment was proposed by Solomon in the seventies \cite{solomon1975random}. In his model, called {\it a~random walk in a~random environment} (RWRE), one first samples the environment by putting random drift independently at every integer, and then the particle moves in such inhomogeneous, random medium. It soon transpired that this additional noise leads to behaviour not observed in the deterministic environment. Various authors described how the distribution of the environment determines such properties of the walk as its transience and asymptotic speed \cite{solomon1975random, alili1999asymptotic}, limit theorems \cite{goldsheid2006simple, kesten1975limit}, or large deviations \cite{dembo1996tail, buraczewski2018precise}. In particular, under suitable distribution of the drift, the walk may be transient, but with sub-linear speed, and no longer satisfy the central limit theorem \cite{kesten1975limit}. This new behaviour is caused, heuristically speaking, by the traps occurring in the environment, i.e.\ sites with unfavourable drift; the particle is forced to make many attempts to cross such a~site and this fact has significant impact on the limiting behaviour of the walk.

The model studied in this article was introduced by Matzavinos, Roitershtein, and Seol in~\cite{matzavinos:2016:random} and is called {\it a~random walk in~a sparse random environment} (RWSRE). The aim is to consider an environment in which the impurities appear not at every site, as it is the case in the RWRE, but are put sparsely on the set of integers. To this end, the environment is sampled by marking some sites by a~two-sided renewal process and putting random drifts only in the marked points. In the unmarked sites the movement of the particle is symmetric. Therefore the RWSRE may be seen as an interpolation between the simple symmetric random walk and the RWRE, and one may expect that, depending on the distribution of the environment, it should manifest properties resembling one or the other. Indeed, this dichotomy was already observed in \cite{buraczewski:2020:random, buraczewski:2019:random, buraczewski:2024:weak} in the context of limit theorems for the position of the walk and the sequence of first passage times. Under suitable assumptions on the distribution of the environment, it is the drift that has major impact on the movement of the particle and the limit theorems resemble results known for the RWRE. However, under different assumptions, which favour long distances between marked points, in most of the sites the walk behaves like a~simple symmetric random walk and this change is visible in the macroscopic scale of the limit theorems.

\begin{figure}[h!]\label{fig:srws}
	\centering
	\begin{subfigure}{.46\textwidth}
		\centering
		\includegraphics[width=\linewidth]{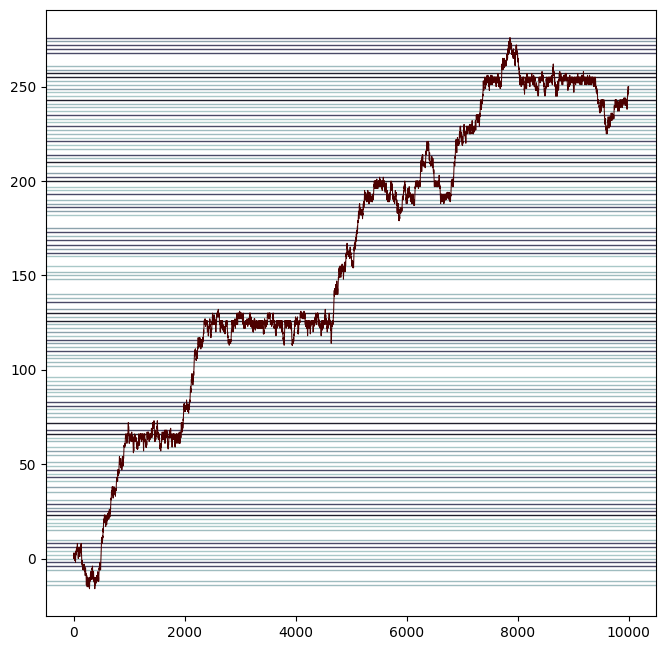}
		\caption{\footnotesize The case of dominating drift: the particle spends most of its time trying to cross sites with unfavourable drift.}
	\end{subfigure}
	\begin{subfigure}{.46\textwidth}
		\centering
		\includegraphics[width=\linewidth]{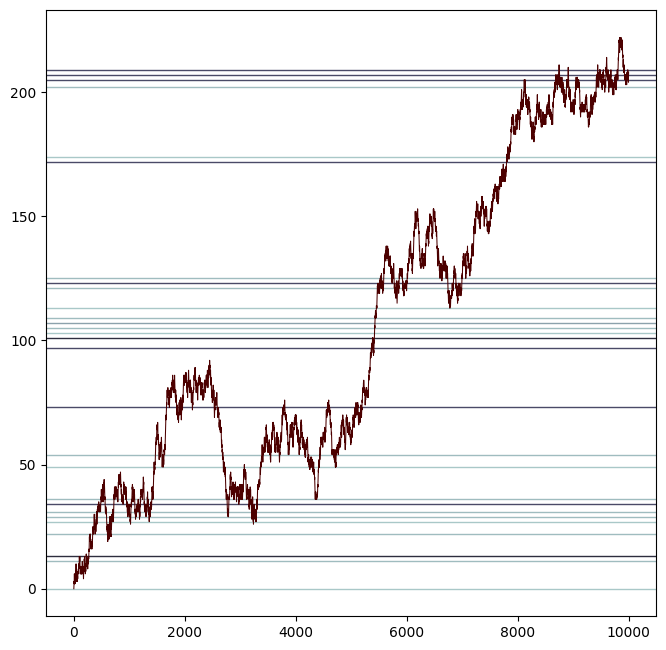}
		\caption{\footnotesize The case of dominating sparsity: in most of the sites the particle performs a~simple symmetric random walk.}
	\end{subfigure}
	\caption{Exemplary trajectories of a~transient RWSRE. Horizontal lines indicate marked sites; the darker the line, the stronger the drift to $-\infty$.}
\end{figure}

The aim of this article is to study the sequence of maximal local times, i.e.\ the amount of time spent by the particle in its favourite sites, in the case of the transient walk in a~sparse random environment. We prove annealed limit theorems for this sequence under two sets of assumptions. In the first case it is the drift that drives the limiting behaviour of the walk, and our results may be seen as a~generalization of those obtained by Dolgopyat and Goldsheid in \cite[Theorem 4]{dolgopyat2010quenched} for the RWRE. However, the techniques used in \cite{dolgopyat2010quenched} were different from those presented here. In this article we follow the method proposed by Kesten et al.\ in \cite{kesten1975limit} when examining the hitting times, that is we rephrase the question posed for the walk into the setting of the associated branching process. This method proves useful both in the case of dominating drift and the complementary case, in which the sparsity of the environment plays the dominant role in determining the limiting behaviour of the walk.

The article is organized as follows: in the remaining part of this section we define the examined model formally. Statement of our main results is given in Section~\ref{sec:local:thms}. Section~\ref{sec:branching} introduces the branching process associated with the walk and presents some of its properties. The proofs of the main theorems are given in Sections \ref{sec:thmA} and \ref{sec:thmB}.

\subsection{Random walk in a~sparse random environment}

Throughout the paper, $\N$ denotes the set of natural numbers (including $0$), $\N_+$ denotes the set of positive natural numbers, and $\Z$ denotes the set of integers. Let $\Omega=(0,1)^\Z$ and let $\F$ be the corresponding cylindrical $\sigma$-algebra. A random element $\omega=(\omega_n)_{n\in\Z}$ of $(\Omega, \F)$ distributed according to a~probability measure $\P$ is called a~{\it random environment}. Let $\X = \Z^{\N}$ be the set of possible paths of a~random walk on $\Z$, with corresponding cylindrical $\sigma$-algebra $\G$. Then any $\omega \in \Omega$ and $i \in \Z$ give rise to a~measure $\Po^i$ on $\X$ such that $\Po^i[X_0=i]=1$ and
\begin{equation}\label{def:Po}
	\Po^i \left[X_{n+1}=j| X_n = k \right] =
	\begin{cases}
		\omega_k \quad & \textnormal{if } j = k+1, \\
		1 - \omega_k & \textnormal{if } j = k-1, \\
		0 & \textnormal{otherwise,}
	\end{cases}
\end{equation}
where $X = (X_n)_{n \in \N} \in \mathcal{X}$. That is, under $\Po^i$, $X$ is a~nearest-neighbour random walk starting from $i$ with transition probabilities given by the sequence $\omega$. In particular, it is a~time-homogeneous Markov chain.

Since the environment itself is random, it is natural to consider a~measure $\PP^i$ on $(\Omega\times \X, \F \otimes \G)$ such that
\begin{equation}\label{def:PP}
	\PP^i \left[F \times G\right] = \int_{F} \Po^i[G] \, \P(d\omega)
\end{equation}
for any $F \in \F, G \in \G$. We write $\Po = \Po^0$ and $\PP = \PP^0$. Observe that under $\PP$ the walk $X$ may exhibit long-time dependencies and thus no longer be a~Markov chain.

The process $X$ defined above is called a~{\it random walk in a~random environment} and was introduced by Solomon \cite{solomon1975random}. A~well studied case is $\omega$ being an i.i.d.\ sequence, which gives rise to a~{\it random walk in i.i.d.\ random environment}.

We will consider a~specific choice of environment that was introduced by Matzavinos, Roitershtein, and Seol in \cite{matzavinos:2016:random}. Consider an i.i.d.\ sequence $((\xi_k,\lambda_k))_{k\in\Z} \in (\N_+ \times [0,1])^\Z$ and define, for any $n,k \in \Z$,
\begin{equation}\label{def:RWSRE}
	S_n = \begin{cases}
		\sum_{j=1}^n \xi_j, \; &n > 0, \\
		0, & n=0, \\
		-\sum_{j={n+1}}^0 \xi_j, & n < 0;
	\end{cases}
	\qquad
	\omega_k = \begin{cases}
		\lambda_{n+1} \: &\textnormal{ if } k = S_n \textnormal{ for some } n\in\Z, \\
		1/2  & \textnormal{ otherwise.}
	\end{cases}
\end{equation}
The random walk evolving in an environment $\omega$ defined by \eqref{def:RWSRE} is called a~{\it random walk in a~sparse random environment}. We refer to the random sites $S_n$ as {\it marked points} and write $(\xi, \lambda)$ for a~generic element of the sequence $((\xi_k,\lambda_k))_{k\in\Z}$. The environment is called {\it moderately sparse} if $\E\xi < \infty$ and {\it strongly sparse} otherwise.

Observe that if $\xi = 1$ almost surely, then we obtain once again a~random walk in i.i.d.\ environment. Otherwise the environment is split into blocks of lengths given by the sequence $(\xi_k)_{k\in\Z}$; within every block the particle performs a~symmetric walk, while the random drift occurs at the endpoints of blocks. Therefore the RWSRE model may be seen as an interpolation between a~simple symmetric random walk and a~walk in i.i.d.\ environment, or as a~generalization of the latter.

We should remark that the model we consider here is slightly different from that defined originally in \cite{matzavinos:2016:random}. That is, due to \eqref{def:RWSRE}, we allow for dependence between the length of the block between marked sites and the drift at its left end, while originally the dependence was allowed for the drift at the right end. This change of convention arises naturally from time reversal coming with the associated branching process which we introduce in Section~\ref{sec:branching}, and appeared also in \cite{buraczewski:2020:random, buraczewski:2019:random}, where annealed limit theorems for the position of the walk were proved.

\medskip

For $k\in\Z$, let
\begin{equation*}
	\rho_k = \frac{1-\lambda_k}{\lambda_k}.
\end{equation*}
The variables $(\rho_k)_{k\in\Z}$, which quantify the drift in the environment, appear naturally when examining the properties of the walk. In particular, as was shown in \cite[Theorem 3.1]{matzavinos:2016:random}, if
\begin{equation}\label{eq:transience}
	\E\log\xi < \infty, \quad \E\log\rho < 0,
\end{equation}
then RWSRE is transient to $+\infty$, $\PP$-almost surely. From now on we will assume that conditions \eqref{eq:transience} are satisfied. Observe that this assumption excludes the case $\P[\rho=1] = 1$, in which $X$ is a~simple symmetric random walk. The case $\P[\rho=0]=1$, in which each marked point is a~barrier that cannot be crossed from the right, is not excluded.

\begin{rem}\label{rem:r}	
	Consider a~function $r : [0,\infty) \to [0,\infty]$ given by 
	\begin{equation}\label{def:r}
		r(x) = \E\rho^x
	\end{equation}
	and assume that $r$ is finite on some interval $[0,\eps]$, $\eps>0$. We will repeatedly make use of the following observation: under conditions \eqref{eq:transience}, either $\P[\rho=0]=1$ and $r(x) = 0$ for all $x>0$, or $r$ is strictly convex on $[0,\eps]$. Moreover, $r'(0) = \E\log\rho < 0$. In particular, there exists at most one $\alpha > 0$ such that $\E\rho^\alpha = 1$; if it exists, then $\E\rho^\gamma < 1$ for $\gamma \in (0,\alpha)$ and $\E\rho^\gamma > 1$ for $\gamma > \alpha$.
\end{rem}
\section{Annealed limit theorems for maximal local time}\label{sec:local:thms}

Consider a~sequence of hitting times
\begin{equation}\label{def:hittingTime}
	T_n = \inf\{k \geq 0: X_k = n \}
\end{equation}
and let, for $k\leq n$,
\begin{equation}\label{def:localTime}
	L_k(n) = |\{m \leq T_n \, : \, X_m = k\}|
\end{equation}
be the {\it local time at $k$}, i.e.\ the number of times the walk visits $k$ before reaching $n$. Our object of interest is the limiting behaviour of maximal local time, that is the variable $\max_{k \leq n} L_k(n)$, as $n \to \infty$. We will present two cases in which an annealed limit theorem holds for this sequence of variables, with Fr\'echet distribution in the limit.

Before the strict formulation of our assumptions let us give a~brief description of the two cases. Roughly speaking, there are two ways in which the maximal local time can be obtained. Either it consists of accumulated visits to some point made during many excursions from a~site with strong unfavourable drift, or is obtained during the time RWSRE behaves like a~symmetric walk, i.e.\ during a~crossing of a~block between marked points. In the first case the amount of time spent in the favourite point should be controlled by the strength of the drift, while in the second case it should depend on the length of the block. Whether the first or the second scenario prevails depends on the joint distribution of $\rho$ and $\xi$, that is on the interplay between the drift and the sparsity.

Recall that we assume \eqref{eq:transience} to hold, so that the walk is transient to $+\infty$. Additionally, we consider the following sets of assumptions:

\medskip
{\bf Assumptions $(A)$:} For some $\alpha \in (0,2)$,
\begin{itemize}
	\setlength\itemsep{0em}
	\item $\E \rho^\alpha = 1$;
	\item $\E \rho^\alpha \log^+\rho < \infty$;
	\item the distribution of $\log\rho$ is non-arithmetic;
	\item $\E \xi^{(\alpha+\delta)\vee 1} < \infty$ for some $\delta > 0$;
	\item $\E \xi^{\alpha} \rho^\alpha < \infty$.
\end{itemize}
\medskip
Recall that a~distribution is non-arithmetic if it is not concentrated on any lattice $c \Z$, $c>0$. Without loss of generality we assume that $\alpha + \delta \leq 2$. In this case the limiting behaviour of maxima is determined mostly by the parameter $\alpha$, that is by properties of $\rho$; it is a~generalization of the result known for the walk in i.i.d.\ environment \cite[Theorem 4]{dolgopyat2010quenched}. We will prove the following:
\begin{thm}\label{thm:A}
	Under assumptions $(A)$, there is a~constant $c_\alpha > 0$ such that for all $x>0$,
	\begin{equation*}
		\lim_{n\to\infty}\PP\left[ \frac{\max_{k \leq n} L_k(n)}{n^{1/\alpha}} > x \right] = 1 - e^{-c_\alpha x^{-\alpha}}. 
	\end{equation*}
\end{thm}
\noindent As may be seen in the proof of Theorem \ref{thm:A}, the first three assumptions in $(A)$ secure the tail asymptotics of certain processes related to the drift in the environment (see Section~\ref{sec:estim} and Lemma~\ref{lem:maxPsi}), while the next two assure that the limiting behaviour of the local times is governed mostly by these processes. We should remark that the non-arithmeticity assumption is technical and we would expect similar statement in the arithmetic case. The crucial assumption in case $(A)$ is $\E\xi^{\alpha+\delta} < \infty$. Different behaviour appears when $\xi$ does not have high enough moments. Consider the following:

\medskip
{\bf Assumptions $(B)$: } For some $\beta \in [1,2)$,
\begin{itemize}
	\setlength\itemsep{0em}
	\item $\P[\xi > x] \sim x^{-\beta}\ell(x)$ for some slowly varying $\ell$;
	\item $\E\rho^{\beta+\delta} < 1$ for some $\delta > 0 $;
	\item $\xi$ and $\rho$ are independent;
	\item if $\beta = 1$, assume $\E\xi < \infty$.
\end{itemize}
\medskip
Here and throughout the article we write $f(x) \sim g(x)$ for two functions $f, g : \R \to \R$ whenever $f(x)/ g(x) \to 1$ as $x \to \infty$. Recall that a~function $\ell$ is slowly varying at infinity if $\ell(c x) \sim \ell(x)$ for any constant $c>0$. In case $(B)$ we may also assume that $\beta + \delta \leq 2$. Observe that we do not assume that there exists $\alpha>0$ such that $\E\rho^\alpha = 1$. However, if it does exist, then the convexity of $r$ defined in \eqref{def:r} implies that $\alpha > \beta$ and $\E\xi^\alpha = \infty$. Under assumptions $(B)$, the asymptotic behaviour of maximal local times is controlled mostly by the distribution of~$\xi$. Since $\xi$ has regularly varying tails, a~good scaling for maxima of $(\xi_n)_{n\in\N}$ is a~sequence $(a_n)_{n\in\N}$ such that
\begin{equation}\label{def:an}
	\lim_{n\to\infty} n \P[\xi > a_n] = 1.
\end{equation}
A typical choice would be $a_n = \inf\{x \geq 0  : \P[\xi > x] \leq 1/n\}$. The sequence $(a_n)_{n\in\N}$ is regularly varying with index $1/\beta$, i.e.
\begin{equation*}
	a_n = n^{1/\beta} \ell_1(n)
\end{equation*}
for some slowly varying function $\ell_1$ (see, for instance, p.\ 15 in \cite{resnick:2013:extreme}). It turns out that $(a_n)_{n\in\N}$ is also a~good scaling for maxima of $L$.
\begin{thm}\label{thm:B}
	Under assumptions $(B)$, there is a~constant $\wt{c}_\beta > 0$ such that for all $x>0$,
	\begin{equation*}
		\lim_{n\to\infty}\PP\left[ \frac{\max_{k \leq n} L_k(n)}{a_n} > x \right] = 1 - e^{-\wt{c}_\beta x^{-\beta}}. 
	\end{equation*}
\end{thm}
The exact forms of constants $c_\alpha, \wt{c}_\beta$ will be given during the proofs.

\medskip

As was mentioned in the introduction, the dichotomy of {\it dominating drift} in case $(A)$ and {\it dominating sparsity} in case $(B)$ was already observed in the limit theorems for first passage times and the position of the walk. Interestingly enough, for the maximal local time the passage from one regime to another occurs under different conditions (see \cite[Theorems~2.2 and~2.6]{buraczewski:2019:random}). That is, assuming that both parameters $\alpha, \beta$ given by the first conditions of cases $(A)$ and $(B)$ exist, the ``critical phase'' for the local times is $\alpha = \beta$, while the change of regimes for the hitting times (and thus also the position of the walk) occurs under $2 \alpha = \beta$. The reason, as may be seen from the arguments used in both cases, is the fact that the first passage times and local times of a~simple symmetric random walk are asymptotically of different orders, while for a~transient RWRE they are comparable.

\section{Auxiliary results}\label{sec:branching}
Instead of examining the local times explicitly, we pass to a~branching process associated with RWSRE. In this section we describe the construction of this process and prove auxiliary lemmas which we will use in both examined cases.

\subsection{Associated branching process}

An important property of a~transient nearest-neighbour random walk on $\Z$ is its duality with a~branching process. Consider a~walk $(X_n)_{n\in\N}$ such that $X_0 = 0$ and $X_n \to \infty$ almost surely, evolving in an environment $\omega = (\omega_k)_{k\in\Z}$. Recall that, for $n \in \N$,
\begin{equation*}
	T_n = \inf\{k \in \N \,:\, X_k = n \}
\end{equation*}
is the first passage time and, for $k \leq n$,
\begin{equation*}
	L_k(n) = |\{m \leq T_n: X_m = k\}|
\end{equation*}
is the local time, i.e.\ the number of times the walk visits site $k$ before reaching $n$. First of all, note that the transience of the walk implies that, almost surely, the walk spends only finite time on the negative half-axis. That is, for any sequence $b_n \to \infty$,
\begin{equation*}
	\frac{\max_{k < 0} L_k(n)}{b_n} \to 0 \quad \PP\textnormal{-a.s.}
\end{equation*}
Therefore, when examining the limit theorems, we may restrict our analysis to the variables $L_k(n)$ for $k \geq 0$.

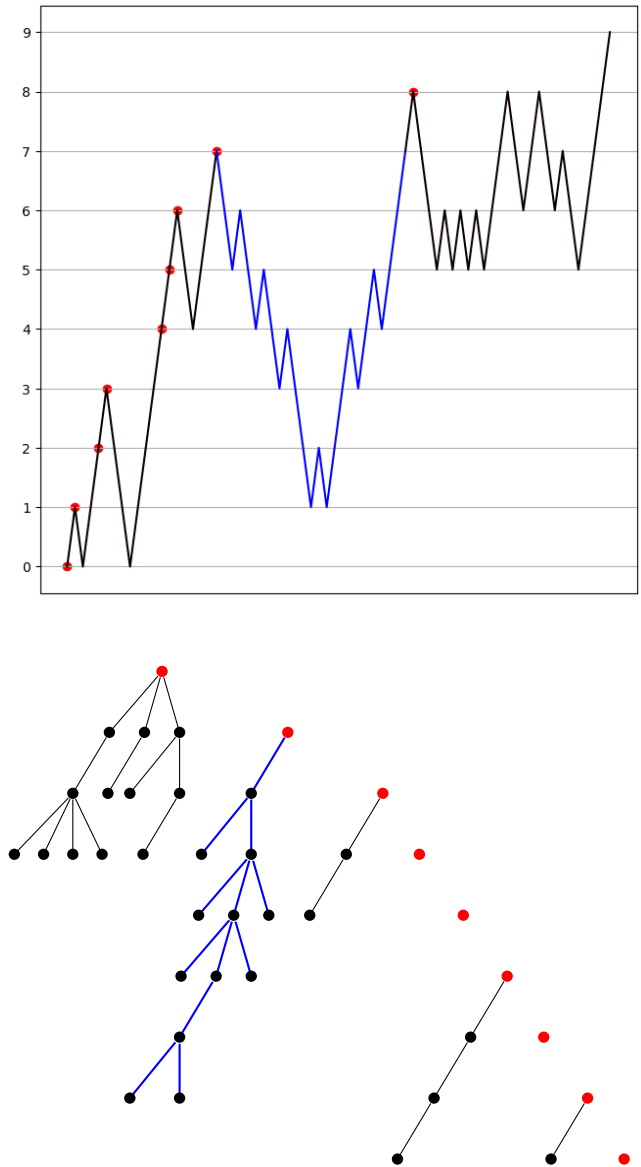
\begin{figure}[h!]
	\begin{minipage}{0.9\linewidth}
		\centering
		\includegraphics[width=0.6\linewidth]{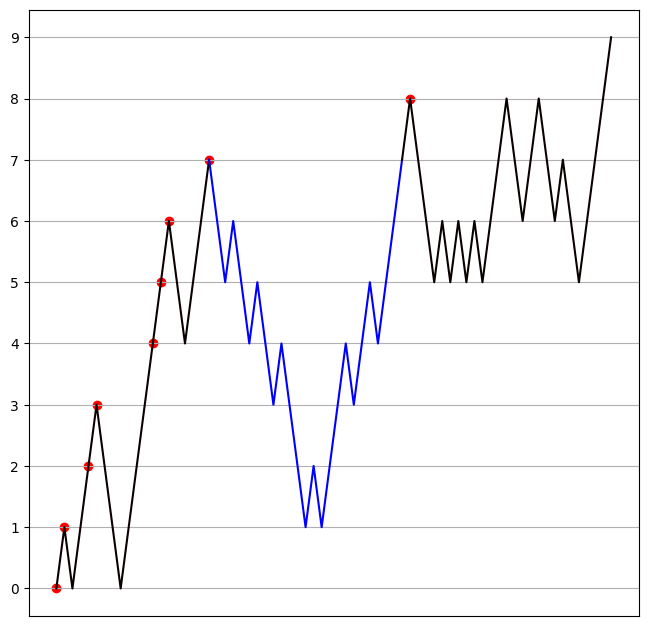}
	\end{minipage}
	\begin{minipage}{0.9\linewidth}
		\centering
		\begin{tikzpicture}
			[every node/.style={fill=black,circle,inner sep=1.5pt}, level distance = 2.1em]
			\node[fill=none]{}
			child[missing] {node{}}
			child { [sibling distance = 1.2em] node[fill=red] {} edge from parent[draw=none]
				child { [sibling distance = 2.5em] node{}
					child { [sibling distance = 1.0em] node{}
						child { [sibling distance = 5.0em] node{}
							child[missing] {node{}}
						}
						child { [sibling distance = 5.0em] node{}
							child[missing] {node{}}
						}
						child { [sibling distance = 5.0em] node{}
							child[missing] {node{}}
						}
						child { [sibling distance = 5.0em] node{}
							child[missing] {node{}}
						}
						child[missing] {node{}}
					}
					child[missing] {node{}}
				}
				child { [sibling distance = 2.5em] node{}
					child { [sibling distance = 5.0em] node{}
						child[missing] {node{}}
					}
					child[missing] {node{}}
				}
				child { [sibling distance = 1.7em] node{}
					child { [sibling distance = 5.0em] node{}
						child[missing] {node{}}
					}
					child { [sibling distance = 2.5em] node{}
						child { [sibling distance = 5.0em] node{}
							child[missing] {node{}}
						}
						child[missing] {node{}}
					}
					child[missing] {node{}}
				}
				child { [sibling distance = 2.5em] node[fill=red, xshift=2.5em] {} edge from parent[draw=none]
					child { [sibling distance = 1.7em] node{} edge from parent[blue, thick]
						child { [sibling distance = 5.0em] node{}
							child[missing] {node{}}
						}
						child { [sibling distance = 1.2em] node{}
							child { [sibling distance = 5.0em] node{}
								child[missing] {node{}}
							}
							child { [sibling distance = 1.2em] node{}
								child { [sibling distance = 5.0em] node{}
									child[missing] {node{}}
								}
								child { [sibling distance = 2.5em] node{}
									child { [sibling distance = 1.7em] node{}
										child { [sibling distance = 5.0em] node{}
											child[missing] {node{}}
										}
										child { [sibling distance = 5.0em] node{}
											child[missing] {node{}}
										}
										child[missing] {node{}}
									}
									child[missing] {node{}}
								}
								child { [sibling distance = 5.0em] node{}
									child[missing] {node{}}
								}
								child[missing] {node{}}
							}
							child { [sibling distance = 5.0em] node{}
								child[missing] {node{}}
							}
							child[missing] {node{}}
						}
						child[missing] {node{}}
					}
					child { [sibling distance = 2.5em] node[fill=red, xshift=2em] {} edge from parent[draw=none]
						child { [sibling distance = 2.5em] node{}
							child { [sibling distance = 5.0em] node{}
								child[missing] {node{}}
							}
							child[missing] {node{}}
						}
						child { [sibling distance = 5.0em] node[fill=red] {} edge from parent[draw=none]
							child[missing] {node{}}
							child { [sibling distance = 5.0em] node[fill=red, xshift=-1em] {} edge from parent[draw=none]
								child[missing] {node{}}
								child { [sibling distance = 2.5em] node[fill=red, xshift=-1em] {} edge from parent[draw=none]
									child { [sibling distance = 2.5em] node{}
										child { [sibling distance = 2.5em] node{}
											child { [sibling distance = 5.0em] node{}
												child[missing] {node{}}
											}
											child[missing] {node{}}
										}
										child[missing] {node{}}
									}
									child { [sibling distance = 5.0em] node[fill=red] {} edge from parent[draw=none]
										child[missing] {node{}}
										child { [sibling distance = 2.5em] node[fill=red, xshift=-1em] {} edge from parent[draw=none]
											child { [sibling distance = 5.0em] node{}
												child[missing] {node{}}
											}
											child { [sibling distance = 5.0em] node[fill=red] {} edge from parent[draw=none]
												child[missing] {node{}}
			}}}}}}}}}; 
		\end{tikzpicture}
		\vspace{10pt}
	\end{minipage}
	\caption{Exemplary path of a~simple walk and corresponding realization of a~branching process. Immigrants (marked in red) correspond to arrivals to new sites. The subtrees correspond to the excursions of the walk; the first excursion from $7$ and its corresponding subtree were marked in blue.}
\end{figure}

The visits to $k \geq 0$ counted by $L_k(n)$ may be split into visits from the left and from the right, that is,
\begin{equation*}
	\begin{split}
		L_k(n) & = |\{m \leq T_n \, : \, X_m = k \}| \\
		& = |\{m \leq T_n \, : \, X_{m-1} = k-1, \, X_m = k\}|
		+ |\{m \leq T_n \, : \, X_{m-1} = k+1, \, X_m = k\}|.
	\end{split}
\end{equation*}
Moreover, since the walk is simple, it makes a~step from $k-1$ to $k$ when it visits site $k$ for the first time. After that, it may make some excursions to the left from $k$; such an excursion always begins with a~step from $k$ to $k-1$ and ends with a~step from $k-1$ to $k$. Therefore, to count all the visits the walk makes to given sites, it is enough to count its steps to the left. That is, for fixed $n \in \N$ and $0 \leq k \leq n$,
\begin{equation*}
	\begin{split}
		L_k(n) & = 1  + |\{m \leq T_n \, : \, X_{m-1} = k, \, X_m = k-1\}|
		+ |\{m \leq T_n \, : \, X_{m-1} = k+1, \, X_m = k\}| \\
		& = 1 + \wt{Z}_{k-1}(n) + \wt{Z}_{k}(n),
	\end{split}
\end{equation*}
where $\wt{Z}_k(n) = |\{m \leq T_n \, : \, X_{m-1} = k + 1, \, X_m = k \}|$ is the number of visits to point $k$ from the right before reaching $n$. The main observation is that the process given by $Z_k(n) = \wt{Z}_{n-k}(n)$ has a~branching structure. Every step from $n-k$ to $n-k-1$ occurs either before the walk discovered the site $n-k+1$, or between consecutive steps from $n-k+1$ to $n-k$. That is,
\begin{equation*}
	Z_{k+1}(n) = \sum_{j=1}^{Z_{k}(n)+1} G_{k}^{(j)}(n),
\end{equation*}
where $G_{k}^{(j)}(n)$, for $j \leq Z_k(n)$, counts the number of steps from $n-k$ to $n-k-1$ between $j$'th and $j+1$'th step from $n-k+1$ to $n-k$, and $G_{k}^{(Z_k(n)+1)}(n)$ counts the number of steps from $n-k$ to $n-k-1$ before the first visit to $n-k+1$. Observe that, due to the strong Markov property of the walk, the variables $G_{k}^{(j)}(n)$ are i.i.d., independent of $Z_{k}(n)$, and have geometric distribution with parameter $\omega_{n-k}$, i.e.
\begin{equation*}
	\Po\left[G_{k}^{(j)}(n) = m\right] = \omega_{n-k}(1-\omega_{n-k})^{m} \quad \textnormal{for } m = 0,1,2,\dots.
\end{equation*}
Therefore, $Z(n) = (Z_k(n))_{0 \leq k \leq n+1}$ is a~branching process in a~random environment with unit immigration; note that we do not count the immigrant, so that $Z_0(n) = 0$. Moreover,
\begin{equation}\label{eq:branchingDist}
	\left(L_k(n)\right)_{0\leq k \leq n} = \left(1 + Z_{n-k+1}(n) + Z_{n-k}(n) \right)_{0 \leq k \leq n}.
\end{equation}

In particular, if $X$ is a~random walk in a~sparse random environment, its associated branching process is a~branching process in a~sparse random environment (BPSRE). If in the above construction we consider the walk stopped upon reaching a~marked point $S_n$, the branching process starts from one immigrant and evolves in the environment divided into blocks of lengths given by $(\xi_{n-k})_{k<n}$; within the blocks the reproduction is given by the law $Geo(1/2)$, while the particles in the $k$'th marked generation are born with the law $Geo(\lambda_{n-k})$. When examining the process $Z$, it is convenient -- and valid, since the environment is given by an i.i.d.\ sequence -- to reverse the enumeration, so that the block lengths are given by $(\xi_k)_{k\in\N}$ and reproduction law in $k$'th marked point is $Geo(\lambda_k)$. The process $Z = (Z_k)_{k\in\N}$ may be then defined formally as follows: for any fixed environment $\omega$, under $\Po$,
\begin{align*}
	Z_0 &= 0, \\
	Z_k &= \sum_{j=1}^{Z_{k-1}+1} G_k^{(j)},
\end{align*}
where the variables $(G_k^{(j)})_{j\in\N}$ are independent of $Z_{k-1}$ and each other, and
\begin{equation*}
	G_k^{(j)} \sim Geo(\omega_k) \quad \textnormal{for} \quad
	\omega_k = \begin{cases}
		\lambda_n \quad & \textnormal{if $k = S_n$ for some $n\in\N$;} \\
		1/2 & \textnormal{otherwise,}
	\end{cases}
\end{equation*}
where $Geo(\omega_k)$ denotes the geometric distribution with parameter $\omega_k$ supported on the set $\{0,1,\dots\}$. Whenever examining a~BPSRE, we will denote the population at marked generations with bold letters, that is, for example, $\Z_n = Z_{S_n}$.

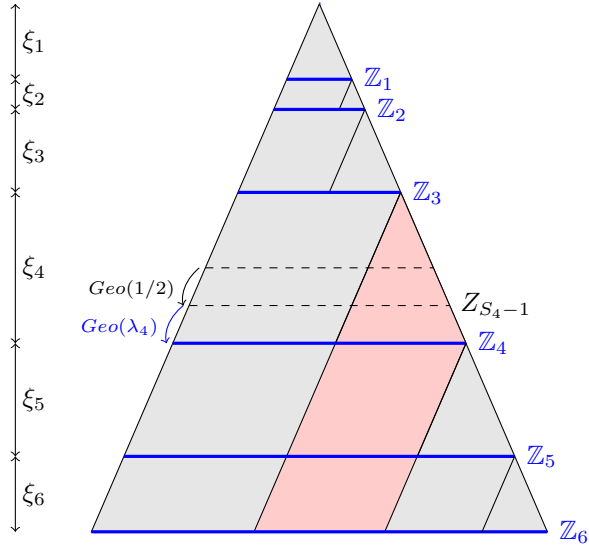
\begin{figure}[h!]
	\centering
	\begin{tikzpicture}
		[every node/.style={inner sep=0pt}]
		\draw[fill=gray!20] (0,0)--(-3.00,-7)--(3.00,-7)--(0,0);
		\draw[fill=red!20] (1.07,-2.50) -- (-0.86,-7.00) -- (0.86,-7.00) -- (1.93,-4.50) -- (1.07,-2.50);
		\draw[very thick, color=blue] (-0.00,-0.00) -- (0.00,-0.00); 
		\draw[very thick, color=blue] (-0.43,-1.00) -- (0.43,-1.00) node[right, xshift=4, text=blue] {\small $\Z_1$};
		\draw[<->](-4.00,-0.00) -- (-4.00,-1.00);
		\node at (-3.75,-0.50) {\small $\xi_1$};
		\draw[very thick, color=blue] (-0.60,-1.40) -- (0.60,-1.40) node[right, xshift=4, text=blue] {\small $\Z_2$};
		\path[draw] (0.43,-1.00) -- (0.26,-1.40);
		\draw[<->](-4.00,-1.00) -- (-4.00,-1.40);
		\node at (-3.75,-1.20) {\small $\xi_2$};
		\draw[very thick, color=blue] (-1.07,-2.50) -- (1.07,-2.50) node[right, xshift=4, text=blue] {\small $\Z_3$};
		\path[draw] (0.60,-1.40) -- (0.13,-2.50);
		\draw[<->](-4.00,-1.40) -- (-4.00,-2.50);
		\node at (-3.75,-1.95) {\small $\xi_3$};
		\draw[very thick, color=blue] (-1.93,-4.50) -- (1.93,-4.50) node[right, xshift=4, text=blue] {\small $\Z_4$};
		\path[draw] (1.07,-2.50) -- (0.21,-4.50);
		\draw[dashed] (-1.71,-4.00) -- (1.71,-4.00) node[right, xshift=4] {\small $Z_{S_4-1}$};
		\draw[dashed] (-1.50,-3.50) -- (1.50,-3.50);
		\draw[->, blue](-1.79,-4.00) .. controls (-2.09,-4.25) and (-2.01,-4.50) .. (-2.01,-4.50) node[left, yshift=6.00, xshift=-3] {\tiny$Geo(\lambda_4)$};
		\draw[->](-1.58,-3.50) .. controls (-1.87,-3.75) and (-1.79,-4.00) .. (-1.79,-4.00) node[left, yshift=6.00, xshift=-3] {\tiny$Geo(1/2)$};
		\draw[<->](-4.00,-2.50) -- (-4.00,-4.50);
		\node at (-3.75,-3.50) {\small $\xi_4$};
		\draw[very thick, color=blue] (-2.57,-6.00) -- (2.57,-6.00) node[right, xshift=4, text=blue] {\small $\Z_5$};
		\path[draw] (1.93,-4.50) -- (1.29,-6.00);
		\draw[<->](-4.00,-4.50) -- (-4.00,-6.00);
		\node at (-3.75,-5.25) {\small $\xi_5$};
		\draw[very thick, color=blue] (-3.00,-7.00) -- (3.00,-7.00) node[right, xshift=4, text=blue] {\small $\Z_6$};
		\path[draw] (2.57,-6.00) -- (2.14,-7.00);
		\draw[<->](-4.00,-6.00) -- (-4.00,-7.00);
		\node at (-3.75,-6.50) {\small $\xi_6$};
	\end{tikzpicture}
	\caption{Schematic picture of the process $Z$. Horizontal blue lines represent marked generations. Within each block between marked generations, the triangular area represents progeny of immigrants that arrived in this block. The coloured region represents process $Y(4)$.}
	\label{fig:bpsre}
\end{figure}

For $k \in \N_+$, we will denote by $Y(k)$ the process counting the progeny of immigrants from the $k$'th block, i.e.\ those arriving at times $S_{k-1}, S_{k-1}+1, \dots S_k-1$. Let, for $j \geq 0$, $Y_j(k)$ denote the number of descendants of these immigrants present in generation $S_{k-1}+j$. Observe that the process $Y(k)$ starts with one immigrant at time $j=0$; it evolves with unit immigration and $Geo(1/2)$ reproduction law up until time $j=\xi_k-1$. The last immigrant arrives at this time, and the particles at time $j=\xi_k$ are born with the law $Geo(\lambda_k)$. From there on the process $Y(k)$ evolves without immigration (see Figure \ref{fig:bpsre}).

We will use the convention that $Y_j(k) = 0$ for $j < 0$, so that
\begin{equation*}
	Z_n = \sum_{k \in \N} Y_{n-S_{k-1}}(k).
\end{equation*}
Observe that the processes $Y(k)$ are independent under~$\Po$ and identically distributed under~$\PP$.

The branching process in a~sparse random environment was studied in \cite{buraczewski:2019:random} for the purpose of proving annealed limit theorems for the first passage times. An important observation is that the transience of the walk implies quick extinctions of the branching process. Let 
\begin{equation*}
	\tau_0 = 0, \quad \tau_n = \inf\{k > \tau_{n-1} \, : \, \Z_k = 0\}
\end{equation*}
be the extinction times (note that we only consider the extinctions at marked generations). Observe that when the extinction occurs, the process starts anew from one immigrant. Thus the sequence $(\tau_{n+1}-\tau_{n})_{n\in\N}$ is i.i.d.\ under $\PP$, and the extinction times split the process $Z$ into independent fragments. The following is Lemma~4.1 from \cite{buraczewski:2019:random}; it implies that the extinctions occur rather often in the case of transient RWSRE.
\begin{lem}\label{lem:tau}
	Assume that $\E\log\rho < 0$ and $\E\log\xi < \infty$. Then $\EE\tau_1 < \infty$. If additionally $\E\rho^\eps < \infty$ and $\E\xi^\eps < \infty$ for some $\eps > 0$, then there exists $c > 0$ such that $\EE e^{c \tau_1} < \infty$.
\end{lem}

\medskip

Observe that due to \eqref{eq:branchingDist} and the fact that the environment is given by an i.i.d.\ sequence we have, for any $n \in \N$,
\begin{equation}\label{eq:Zk-to-Lk}
	\max_{0 \leq k \leq S_n} L_{k}(S_n) \od 1 + \max_{0 \leq k \leq S_n}(Z_k + Z_{k+1}).
\end{equation}
Therefore, to obtain limit theorems for the sequence of maximal local times along the marked points, one may examine the maximal generations of the corresponding branching process. We conclude this section by showing that in the setting of moderately sparse environment this is sufficient also to obtain annealed limit theorems for the sequence $(\max_{k\leq n} L_k(n))_{n\in\N}$. Recall that $(a_n)_{n\in\N}$ given by \eqref{def:an} is regularly varying with index $1/\beta$.

\begin{lem}\label{lem:Sn-to-n}
	Assume that $\E\xi < \infty$. If there exist constants $c>0$, $\gamma > 0$ and a~sequence $(b(n))_{n\in\N}$ which is regularly varying with index $1/\gamma$ such that for every $x > 0$,
	\begin{equation*}
		\lim_{n\to\infty}\PP\left[ \frac{\max_{k \leq S_n} L_k(S_n)}{b(n)} > x \right] = 1 - e^{-c x^{-\gamma}},
	\end{equation*}
	then for every $x > 0$,
	\begin{equation*}
		\lim_{n\to\infty}\PP\left[ \frac{\max_{k \leq n} L_k(n)}{b(n)} > x \right] = 1 - e^{-(c/\E\xi) x^{-\gamma}}.
	\end{equation*}
\end{lem}
\begin{proof}
	Denote, for $n \in \N$,
	\begin{equation*}
		\nu_n = \inf\{k>0 \, : \,S_k > n \}.
	\end{equation*}
	The law of large numbers for renewal processes (c.f.\ \cite[Theorem 2.6.1]{durrett2019probability}) guarantees that $\P$-almost surely
	\begin{equation*}
		\frac{\nu_n}{n} \xrightarrow{n\to\infty} \frac{1}{\E\xi}.
	\end{equation*}
	Denote, for $m\in\N$, $M(m) = \max_{k\leq S_m}L_k(S_m)$. Since $S_{\nu_n-1} \leq n < S_{\nu_n}$, we have, for any $\eps \in (0,1/\E\xi)$,
	\begin{equation*}
		\begin{split}
			\PP\left[ b(n)^{-1} \max_{0 \leq k < n} L_k(n) > x\right]
			& \geq \PP\left[ b(n)^{-1} M(\nu_n - 1) > x\right] \\
			& \geq \PP\left[ b(n)^{-1} M(n (1/\E\xi - \eps) - 1) > x\right] - \PP\left[ |1/\E\xi - \nu_n/n| > \eps \right] \\
			& \xrightarrow{n\to\infty} 1 - \exp(-c (1/\E\xi - \eps) x^{-\gamma}),
		\end{split}
	\end{equation*}
	where we used the fact that
	\begin{equation*}
		\frac{b(n(1/\E\xi - \eps)-1)}{b(n)} \to (1/\E\xi - \eps)^{1/\gamma}
	\end{equation*}
	since $b(n)$ is regularly varying. Similarly,
	\begin{equation*}
		\begin{split}
			\PP\left[ b(n)^{-1} \max_{0 \leq k < n} L_k(n) > x\right]
			& \leq \PP\left[ b(n)^{-1} M(\nu_n) > x\right] \\
			& \leq \PP\left[ b(n)^{-1} M(n (1/\E\xi + \eps)) > x\right] + \PP\left[ |1/\E\xi - \nu_n/n| > \eps \right] \\
			& \xrightarrow{n\to\infty} 1 - \exp(-c (1/\E\xi + \eps) x^{-\gamma}),
		\end{split}
	\end{equation*}
	which ends the proof since $\eps\in(0,1/\E\xi)$ is arbitrary.
\end{proof}

\subsection{Estimates of the processes related to the environment}\label{sec:estim}

Define
\begin{equation}\label{def:Rabarbar}
	\Rb_n = 1 + \rho_n + \rho_n\rho_{n+1} + \dots = \sum_{k=n-1}^\infty \Pi_{n,k},
\end{equation}
where $\Pi_{n,k} = \prod_{j=n}^k \rho_j$ if $n \leq k$ and $\Pi_{n,k} = 1$ otherwise. Then the following relation holds:
\begin{equation}\label{eq:Rbrec}
	\Rb_n = 1 + \rho_n \Rb_{n+1}.
\end{equation}
Moreover, the sequence $(\Rb_n)_{n\in\Z}$ is stationary under $\P$. Observe that if $\E\rho^\gamma < 1$ for some $\gamma > 0$, then $\E\Rb_1^\gamma < \infty$ (see the proof of Lemma~2.3.1 in \cite{buraczewski:2016:power}), whereas under $(A)$, the distribution of $\rho$ satisfies the assumptions of the Kesten-Goldie theorem \cite[Theorem 2.4.4]{buraczewski:2016:power}, thus
\begin{equation*}
	\P[\Rb_1 > x] \sim c x^{-\alpha}
\end{equation*}
for some constant $c$. Therefore
\begin{equation}\label{eq:Rabarbarleq}
	\P[\Rb_1 > x] \leq C_\gamma x^{-\gamma} \quad \textnormal{for some $C_\gamma < \infty$ and all $x > 0$,}
\end{equation}
whenever either $\E\rho^\gamma < 1$, or $\E\rho^\gamma = 1$ and the Kesten-Goldie theorem holds for $\Rb_1$. As can be seen in the proofs of Lemma~6 in \cite{kesten1975limit} and Lemma~5.6 in \cite{buraczewski:2019:random}, in the case of dominating drift it is $\Rb_1$ from whom the total population of the process $Z$ (which corresponds to the first passage times of the walk) inherits its annealed tail behaviour.

Let, for $m \in \N$, the {\it potential} $\Psi$ be defined as
\begin{equation}\label{def:psi}
	\Psi_{m,k} = \Pi_{m,n} \quad \textnormal{for } k \in [S_n, S_{n+1}).
\end{equation}
As we will see, maxima of the potential determine the limiting behaviour of maximal generation of $Z$ in the same way as $\Rb_1$ determines the asymptotics of the total population. Let
\begin{equation}\label{def:Mpsi}
	M_{\Psi,m} = \max_{k \geq S_m-1} (\Psi_{m,k} + \Psi_{m,k+1}).
\end{equation}
Then the sequence $(M_{\Psi, m})_{m\in\N}$ is stationary under $\P$; denote by $M_{\Psi}$ its generic element. Observe that
\begin{equation*}
	M_{\Psi,1} \leq 2 \max_{k\geq S_1-1} \Psi_{1,k} = 2 \max_{n\geq 0} \Pi_{1,n} \leq 2 \Rb_1,
\end{equation*}
thus
\begin{equation}\label{eq:psimoment}
	\E M_{\Psi}^\gamma < \infty \quad \textnormal{ whenever } \E\rho^\gamma < 1.
\end{equation}

\subsection{Auxiliary lemmas}

The following lemma, concerning a~classic Galton-Watson process, will be used repeatedly to estimate the growth of BPSRE in the unmarked generations.

\begin{lem}\label{lem:maxGW}
	Let $(X_n)_{n \in \N}$ be a~Galton-Watson process with $X_0 = x_0$, reproduction law $Geo(1/2)$, and no immigrants. Let $(\bar{X}_n)_{n \in \N}$ be an analogous process with one immigrant arriving at each generation (recall that the immigrant is not counted by the process). Then the following hold for any $N \in \N$:
	\begin{equation}\label{eq:GWnoimm}
		\EE \left[ \max_{k\leq N} (X_k - x_0)^2\right] \leq 8 N x_0,
	\end{equation}
	\begin{equation}\label{eq:GWimmvar}
		\EE \left[ \max_{k\leq N} \bar{X}_k^2 \right] \leq 16 (N^2 + Nx_0 + x_0^2).
	\end{equation}
\end{lem}

\begin{proof}
	Since the reproduction mean of the examined Galton-Watson processes is $1$, $(X_n)_{n\in\N}$ is a~martingale with respect to its canonical filtration. In particular, $\EE X_n = x_0$ for every $n\in\N$. Moreover, Doob's maximal inequality implies
	$$\EE \left[ \max_{k \leq N} (X_k - x_0)^2 \right] \leq 4 \EE (X_N - x_0)^2 = 4 \var X_N.$$
	A~standard calculation (see, for instance, formula~(2) on p.~4 in~\cite{athreya}) gives
	$$ \var X_N = 2 N x_0,$$
	which implies \eqref{eq:GWnoimm}.
	
	Observe that $\bar{X}_n = X'_n + I_n$, where $X'$ denotes the descendants of the initial $x_0$ particles and $I$ denotes the progeny of immigrants. The processes $I$ and $X'$ are independent, and $X'$ has the same distribution as $X$. Moreover, the process $(\bar{X}_n)_{n\in\N}$ is a~non-negative submartingale, thus by Doob's maximal inequality,
	\begin{equation*}
		\EE \left[ \max_{k\leq N} \bar{X}_k^2 \right] \leq 4 \EE \left[ \bar{X}_N^2 \right] = 4 \left(\var X'_N + \var I_N + (\EE X'_N + \EE I_N)^2 \right).
	\end{equation*}
	
	We have already examined the mean and variance of $X'_N$. To calculate moments of $I_N$, we may express $I$ as a~sum of independent copies of $X$. Alternatively, we may use the duality of $I$ and a~simple symmetric random walk. It implies that $I_N$ is equal in distribution to the number of times the walk hits $0$ from the right when crossing the interval $[0,N+1]$ for the first time. By the classic gambler's ruin problem, the probability that the walk passes from $0$ to $N+1$ without returning to $0$ from the right is $1/(N+1)$. Therefore $I_N \sim Geo(1/(N+1))$, from which it follows that
	\begin{equation*}
		\EE I_N = N, \quad \var I_N = N^2 + N.
	\end{equation*}
	Hence
	\begin{equation*}
		\EE \left[ \bar{X}_N^2 \right] = 2Nx_0 + N^2 + N + (x_0 + N)^2 \leq 4(N^2 + Nx_0 + x_0^2),
	\end{equation*}
	which ends the proof of \eqref{eq:GWimmvar}.
	
\end{proof}

The next two lemmas will be of use to us both under assumptions $(A)$ and $(B)$. Therefore we consider the following set of assumptions:

\medskip
{\bf Assumptions $(\Gamma)$:} for some $\gamma \in (0,2]$,
\begin{itemize}
	\setlength\itemsep{0em}
	\item $\E\rho^\gamma \leq 1$ and \eqref{eq:Rabarbarleq} holds,
	\item $\E\xi^{\gamma/2} < \infty$,
	\item $\E \xi^{\gamma/2} \rho^\gamma < \infty$.
\end{itemize}
\medskip
Recall that \eqref{eq:transience} remains our standing assumption. As discussed in Section~\ref{sec:estim}, under $(A)$ the Kesten-Goldie theorem implies \eqref{eq:Rabarbarleq} with $\gamma=\alpha$, while under $(B)$, \eqref{eq:Rabarbarleq} holds with $\gamma=\beta+\delta$. In particular, assumptions $(A)$ imply $(\Gamma)$ with $\gamma = \alpha$ and assumptions $(B)$ imply $(\Gamma)$ with $\gamma = \beta + \delta$.

\smallskip

In the setting of the BPSRE, let $U_n$ be the progeny of the first immigrant residing in generation $n$, with the convention $U_0 = 1$, and denote $\U_n = U_{S_n}$. For fixed $N \in \N$, let $U^{(k)}$ for $k = 1,\dots, N$ be copies of the process $U = (U_n)_{n\in\N}$, evolving in the same environment and independent under $\Po$. That is, $(\sum_{k=1}^N U_n^{(k)})_{n\in\N}$ is a~BPSRE with $N$ initial particles evolving without immigration. Although the first part of the following lemma is analogous to results presented in \cite[Lemma~3]{kesten1975limit} and \cite[Lemma~5.6]{buraczewski:2019:random}, we provide the full proof as it gives some insight into the properties of the process $U$.

\begin{lem}\label{lem:Nimmigrants}
	Assume $(\Gamma)$. Then for some constant $C_1$,
	\begin{equation}\label{eq:Nimmsum}
		\PP \left[ \sum_{k=1}^N \sum_{n\geq 0} \U_n^{(k)} > x \right] \leq C_1 N^\gamma x^{-\gamma},
	\end{equation}
	\begin{equation}\label{eq:Nimmsumdiff}
		\PP \left[ \sum_{n\geq 0} \left|\sum_{k=1}^N  \U_n^{(k)} - N\Pi_{1,n}\right| > x \right] \leq C_1 N^{\gamma/2} x^{-\gamma}. 
	\end{equation}
	Moreover,
	\begin{equation}\label{eq:Nimmmax}
		\PP \left[ \max_{n \geq 1} \sum_{k=1}^N U^{(k)}_n > x \right] \leq C_1 N^\gamma x^{-\gamma},
	\end{equation}
	\begin{equation}\label{eq:Nimmmaxdiff}
		\PP \left[ \sum_{n \geq 1} \sum_{k=1}^N \max_{S_{n-1} \leq j < S_{n}} |U^{(k)}_j-\U^{(k)}_{n-1}| > x \right] \leq C_1 N^{\gamma/2} x^{-\gamma}. 
	\end{equation}
\end{lem}

\begin{proof}
	For fixed $n \geq 1$, under $\Po$,
	\begin{equation*}
		\U_{n} \od \sum_{k=1}^{U_{S_{n}-1}} G^{(n)}_k,
	\end{equation*}
	where $G^{(n)}_k$ are random variables with law $Geo(\lambda_n)$, independent of $U_{S_{n}-1}$ and each other. In particular,
	\begin{equation*}
		\Eo G_k^{(n)} = \rho_n, \quad \varo G_k^{(n)} = \rho_n + \rho_n^2.
	\end{equation*}
	Since in generations $S_{n-1}+1, \dots S_n-1$ the process evolves with offspring distribution $Geo(1/2)$, the calculation performed in the proof of Lemma~\ref{lem:maxGW} gives
	\begin{equation*}
		\Eo [U_{S_{n}-1} | \U_{n-1}] = \U_{n-1} \quad {\rm and } \quad \varo(U_{S_{n}-1} | \U_{n-1}) = 2(\xi_n-1) \U_{n-1}.
	\end{equation*}
	This in turn implies
	\begin{equation}\label{Uevolution}
		\begin{split}
			\Eo[\U_n | \U_{n-1}] &= \rho_n \U_{n-1},\\
			\Eo[(\U_n - \rho_n\U_{n-1})^2 | \U_{n-1}] &= (\rho_n - \rho_n^2 + 2 \xi_n\rho_n^2)\U_{n-1}.
		\end{split}
	\end{equation}
	In particular $\Eo \U_n = \Pi_{1,n}$.
	
	Observe that the processes $U^{(k)}$ evolve without immigration and the extinction time of each $U^{(k)}$ is stochastically dominated by $\tau_1$, which is finite $\PP$-a.s.\ by Lemma~\ref{lem:tau}. In particular, with probability~$1$ the series
	\begin{equation*}
		\sum_{k=1}^N\sum_{n\geq 0} \U_n^{(k)}
	\end{equation*}
	is indeed a~finite sum.	Recall the sequence $\Rb$ defined in \eqref{def:Rabarbar} and observe that, by \eqref{eq:Rbrec},
	\begin{equation*}
		\begin{split}
			\sum_{k=1}^N \sum_{n \geq 0} \U_n^{(k)}
			&= \sum_{k=1}^N \sum_{n \geq 0} \U_n^{(k)} ( \Rb_{n+1} - \rho_{n+1}\Rb_{n+2} ) \\
			& = \sum_{n \geq 1} \left(\sum_{k=1}^N (\U_{n}^{(k)} - \rho_{n}\U^{(k)}_{n-1})\right)\Rb_{n+1} + N \Rb_1
		\end{split}
	\end{equation*}
	and thus
	\begin{equation*}
		\sum_{n \geq 0} \left( \sum_{k=1}^N \U_n^{(k)} - N\Pi_{1,n} \right) = \sum_{n \geq 1} \left(\sum_{k=1}^N (\U_{n}^{(k)} - \rho_{n}\U^{(k)}_{n-1})\right)\Rb_{n+1}.
	\end{equation*}
	Moreover, representing each $\U_n^{(k)} - \Pi_{1,n}$ as a~telescoping sum, we obtain
	\begin{equation*}
		\begin{split}
			\sum_{n\geq 0} \left| \sum_{k=1}^N \U_n^{(k)} - N\Pi_{1,n} \right| &= \sum_{n\geq 0} \left| \sum_{j=1}^n \sum_{k=1}^N \left(\U_j^{(k)} \Pi_{j+1,n} - \U_{j-1}^{(k)} \Pi_{j,n}\right) \right| \\
			& \leq \sum_{n\geq 0} \sum_{j=1}^n \left| \sum_{k=1}^N \U_j^{(k)} - \rho_j \U_{j-1}^{(k)} \right| \Pi_{j+1,n} \\
			& = \sum_{j\geq1} \left| \sum_{k=1}^N \U_j^{(k)} - \rho_j \U_{j-1}^{(k)} \right| \Rb_{j+1},
		\end{split}
	\end{equation*}
	therefore
	\begin{equation*}
		\begin{split}
			\PP \left[ \sum_{n \geq 0} \left| \sum_{k=1}^N \U_n^{(k)} - N\Pi_{1,n} \right| > x \right]
			& \leq \PP \left[ \sum_{n \geq 1} \left|\sum_{k=1}^N \U_{n}^{(k)} - \rho_{n}\U^{(k)}_{n-1}\right|\Rb_{n+1} > x \right]
		\end{split}
	\end{equation*}
	and
	\begin{equation*}
		\begin{split}
			\PP \left[ \sum_{k=1}^N \sum_{n\geq 0} \U_n^{(k)} > x \right]
			& \leq \PP \left[ \sum_{n \geq 1} \left|\sum_{k=1}^N (\U_{n}^{(k)} - \rho_{n}\U^{(k)}_{n-1})\right|\Rb_{n+1} > x/2 \right] + \PP [N \Rb_1 > x/2].
		\end{split}
	\end{equation*}	
	
	Observe that for any $n \geq 1$, $\Rb_{n+1}$ is independent of $(\U_{n}^{(k)} - \rho_{n}\U^{(k)}_{n-1})$. Since $\sum_{n\geq 1} n^{-2} \leq 2$, we have for any $x>0$,
	\begin{multline*}
		\PP \left[ \sum_{n \geq 1} \left|\sum_{k=1}^N (\U_{n}^{(k)} - \rho_{n}\U^{(k)}_{n-1})\right|\Rb_{n+1} > x \right]
		\leq \sum_{n \geq 1} \PP \left[ \left|\sum_{k=1}^N (\U_{n}^{(k)} - \rho_{n}\U^{(k)}_{n-1})\right|\Rb_{n+1} > x/2n^2 \right] \\
		\begin{aligned}
			& = \sum_{n \geq 1} \int_{[0,\infty)} \P[\Rb_{n+1} > x/2tn^2] \PP \left[ \left|\sum_{k=1}^N (\U_{n}^{(k)} - \rho_{n}\U^{(k)}_{n-1})\right| \in dt\right] \\
			& \leq C_\gamma \, \sum_{n \geq 1} \int_{[0,\infty)} (x/2tn^2)^{-\gamma} \PP \left[ \left|\sum_{k=1}^N (\U_{n}^{(k)} - \rho_{n}\U^{(k)}_{n-1})\right| \in dt\right] \\
			& = 2^{\gamma} C_\gamma \, x^{-\gamma} \sum_{n \geq 1} n^{2\gamma} \EE \left|\sum_{k=1}^N (\U_{n}^{(k)} - \rho_{n}\U^{(k)}_{n-1})\right|^\gamma,
		\end{aligned}
	\end{multline*}
	where the second inequality follows from \eqref{eq:Rabarbarleq}.
	
	The relations \eqref{Uevolution} imply that for any fixed $n$, under $\Po$, $\sum_{k=1}^N (\U_{n}^{(k)} - \rho_{n}\U^{(k)}_{n-1})$ is a~sum of independent centered variables; in particular, using formulae \eqref{Uevolution}, we obtain
	\begin{equation*}
		\begin{split}
			\Eo \left(\sum_{k=1}^N (\U_{n}^{(k)} - \rho_{n}\U^{(k)}_{n-1})\right)^2
			&= N \Eo (\U_{n} - \rho_{n}\U_{n-1})^2 \\
			& = N (\rho_n + 2\xi_n\rho_n^2 - \rho_n^2) \Eo\U_{n-1} \\
			& = N (\rho_n + 2\xi_n\rho_n^2 - \rho_n^2)\Pi_{1,n-1}.
		\end{split}
	\end{equation*}
	Therefore, conditional Jensen's inequality and subadditivity of the function $x \mapsto x^{\gamma/2}$ (recall $\gamma \leq 2$) give
	\begin{equation*}
		\begin{split}
			\sum_{n \geq 1} n^{2\gamma} \EE \left|\sum_{k=1}^N (\U_{n}^{(k)} - \rho_{n}\U^{(k)}_{n-1})\right|^\gamma
			& \leq \sum_{n \geq 1} n^{2\gamma} \E \left(\Eo\left(\sum_{k=1}^N (\U_{n}^{(k)} - \rho_{n}\U^{(k)}_{n-1})\right)^2\right)^{\gamma/2} \\
			& = N^{\gamma/2} \sum_{n \geq 1} n^{2\gamma} \E ((\rho_n + 2\xi_n\rho_n^2 - \rho_n^2)\Pi_{1,n-1})^{\gamma/2} \\
			& \leq N^{\gamma/2} \sum_{n \geq 1} n^{2\gamma} (\E\rho^{\gamma/2} + 2\E\xi^{\gamma/2}\rho^\gamma)(\E\rho^{\gamma/2})^{n-1}.
		\end{split}
	\end{equation*}
	Observe that under $(\Gamma)$, by Remark \ref{rem:r}, $\E\rho^{\gamma/2} < 1$. In particular, the series is convergent and thus for a~constant $C>0$,
	\begin{equation*}
		\begin{split}
			\PP \left[ \sum_{n \geq 1} \left|\sum_{k=1}^N (\U_{n}^{(k)} - \rho_{n}\U^{(k)}_{n-1})\right|\Rb_{n+1} > x \right]
			& \leq 2^{\gamma} C_\gamma \, x^{-\gamma} \sum_{n \geq 1} n^{2\gamma} \EE \left|\sum_{k=1}^N (\U_{n}^{(k)} - \rho_{n}\U^{(k)}_{n-1})\right|^\gamma \\
			& \leq C N^{\gamma/2} x^{-\gamma},
		\end{split}
	\end{equation*}
	which proves \eqref{eq:Nimmsumdiff}. Invoking \eqref{eq:Rabarbarleq} once again, we conclude that
	\begin{equation*}
		\begin{split}
			\PP \left[ \sum_{k=1}^N \sum_{n\geq 0} \U_n^{(k)} > x \right]
			& \leq \PP \left[ \sum_{n \geq 1} \left(\sum_{k=1}^N (\U_{n}^{(k)} - \rho_{n}\U^{(k)}_{n-1})\right)\Rb_{n+1} > x/2 \right] + \PP [N \Rb_1 > x/2] \\
			& \leq C N^{\gamma/2} (x/2)^{-\gamma} + C_\gamma N^\gamma (x/2)^{-\gamma},
		\end{split}
	\end{equation*}
	which proves \eqref{eq:Nimmsum}.
	
	To show \eqref{eq:Nimmmax}, decompose
	\begin{multline*}
		\PP \left[ \max_{j \geq 0} \sum_{k=1}^N U^{(k)}_n > x \right] = \PP \left[ \max_{n \geq 0} \max_{S_n \leq j < S_{n+1}} \sum_{k=1}^N U^{(k)}_j > x \right] \\
		\begin{aligned}
			& \leq \PP \left[ \sum_{n \geq 0} \sum_{k=1}^N \max_{S_n \leq j < S_{n+1}} U^{(k)}_j > x \right] \\
			& \leq \PP \left[ \sum_{n \geq 0} \sum_{k=1}^N \left( \U^{(k)}_n + \max_{S_n \leq j < S_{n+1}} |U^{(k)}_j-\U^{(k)}_n|\right) > x \right] \\
			& \leq \PP \left[ \sum_{k=1}^N \sum_{n \geq 0} \U^{(k)}_n > x/2 \right]
			+ \PP \left[ \sum_{n \geq 1} \sum_{k=1}^N \max_{S_{n-1} \leq j < S_{n}} |U^{(k)}_j-\U^{(k)}_{n-1}| > x/2 \right],
		\end{aligned}
	\end{multline*}
	which ensures that \eqref{eq:Nimmmax} follows from \eqref{eq:Nimmsum} and \eqref{eq:Nimmmaxdiff}. To show \eqref{eq:Nimmmaxdiff}, note that, under $\Po$, $(U_j)_{S_{n-1}\leq j < S_n}$ is a~Galton-Watson process with $\U_{n-1}$ initial particles and $Geo(1/2)$ reproduction law, evolving without immigration. Thus, by~Lemma~\ref{lem:maxGW},
	\begin{equation*}
		\Eo \left[ \max_{S_{n-1} \leq j < S_{n}} |U_j-\U_{n-1}|^2 \right] \leq 8\xi_{n} \Eo\U_{n-1} = 8\xi_{n} \Pi_{1,n-1}. 
	\end{equation*}
	Therefore
	\begin{multline*}
		\PP  \left[ \sum_{n \geq 1} \sum_{k=1}^N \max_{S_{n-1} \leq j < S_{n}} |U^{(k)}_j-\U^{(k)}_{n-1}| > x/2 \right]
		\leq \sum_{n \geq 1} \PP \left[ \sum_{k=1}^N \max_{S_{n-1} \leq j < S_n} |U^{(k)}_j-\U^{(k)}_{n-1}| > x/4n^2 \right] \\
		\begin{aligned}
			\leq & \sum_{n \geq 1} (x/4n^2)^{-\gamma} N^{\gamma/2} \E \left(\Eo \max_{S_{n-1} \leq j < S_{n}} |U_j-\U_{n-1}|^2\right)^{\gamma/2} \\
			\leq & N^{\gamma/2} x^{-\gamma} \sum_{n \geq 1} (4n)^{2\gamma} 8^{\gamma/2} \E\xi^{\gamma/2} (\E\rho^{\gamma/2})^{n-1} \\
			= & C' N^{\gamma/2} x^{-\gamma},
		\end{aligned}
	\end{multline*}
	for a~constant $C'>0$, which proves \eqref{eq:Nimmmax} and \eqref{eq:Nimmmaxdiff}.
\end{proof}

Let $Y = (Y_n)_{n \in\N}$ be a~copy of the process $(Y_n(1))_{n\in\N}$. That is, $Y$ starts with one immigrant in generation $0$ and for the next $\xi_1 - 1$ generations evolves as a~Galton-Watson process with unit immigration and reproduction law $Geo(1/2)$. The last immigrant arrives in generation $\xi_1-1$; particles there reproduce with distribution $Geo(\lambda_1)$, giving birth to the first marked generation $\Y_1 = Y_{S_1}$. From there on the process evolves without immigration, with particles in each marked generation $\Y_n = Y_{S_n}$ being born with $Geo(\lambda_n)$ distribution, and $Geo(1/2)$ in consecutive blocks of lengths given by $\xi_n-1$ for $n\geq 2$.

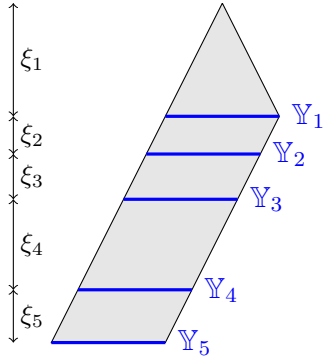
\begin{figure}[h!]
	\label{fig:bpsreYxi}
	\centering
	\begin{tikzpicture}
		[every node/.style={inner sep=0pt}]
		\draw[fill=gray!20] (0.00,-0.00) -- (-2.25,-4.50) -- (-0.75,-4.50) -- (0.75,-1.50) -- (0.00,-0.00);
		\draw[very thick, color=blue] (-0.75,-1.50) -- (0.75,-1.50)  node[right, text=blue, xshift=4] {$\Y_1$};
		\draw[<->](-2.75,-0.00) -- (-2.75,-1.50);
		\node at (-2.50,-0.75) {\small $\xi_1$};
		\draw[very thick, color=blue] (-1.00,-2.00) -- (0.50,-2.00)  node[right, text=blue, xshift=4] {$\Y_2$};
		\draw[<->](-2.75,-1.50) -- (-2.75,-2.00);
		\node at (-2.50,-1.75) {\small $\xi_2$};
		\draw[very thick, color=blue] (-1.30,-2.60) -- (0.20,-2.60)  node[right, text=blue, xshift=4] {$\Y_3$};
		\draw[<->](-2.75,-2.00) -- (-2.75,-2.60);
		\node at (-2.50,-2.30) {\small $\xi_3$};
		\draw[very thick, color=blue] (-1.90,-3.80) -- (-0.40,-3.80)  node[right, text=blue, xshift=4] {$\Y_4$};
		\draw[<->](-2.75,-2.60) -- (-2.75,-3.80);
		\node at (-2.50,-3.20) {\small $\xi_4$};
		\draw[very thick, color=blue] (-2.25,-4.50) -- (-0.75,-4.50)  node[right, text=blue, xshift=4] {$\Y_5$};
		\draw[<->](-2.75,-3.80) -- (-2.75,-4.50);
		\node at (-2.50,-4.15) {\small $\xi_5$};
	\end{tikzpicture}
	\caption{Schematic picture of the process $Y$. Horizontal blue lines represent marked generations. The immigrants arrive only in the first block.}
\end{figure}

\begin{lem}\label{lem:immmax}
	Assume $(\Gamma)$. Then for some constant $C_2$,
	\begin{equation}\label{eq:immmaxsparse}
		\PP \left[ \max_{n \geq 1} Y_n > x \right] \leq C_2 x^{-\gamma} \left( \E \left(\Eo Y_{\xi_1-1}^2\right)^{\gamma/2} + \EE \Y_1^\gamma \right).
	\end{equation}
	If additionally $\E\xi^\gamma < \infty$ and $\E\xi^\gamma\rho^\gamma < \infty$, then for some constant $C_3$,
	\begin{equation}\label{eq:immmaxdrift}
		\PP \left[ \max_{n \geq 1} Y_n > x \right] \leq C_3 x^{-\gamma}.
	\end{equation}
\end{lem}

\begin{proof}
	We have
	\begin{equation}\label{eq:Ydecomp}
		\PP \left[ \max_{n \geq 1} Y_n > x \right]
		\leq \PP \left[ \max_{n < S_1} Y_n > x \right] 
		+ \PP \left[ \max_{n \geq S_1} Y_n > x \right].
	\end{equation}
	
	For the first $\xi_1-1$ generations $Y$ evolves as a~Galton-Watson process with unit immigration and reproduction law $Geo(1/2)$, therefore $(Y_n^2)_{n < S_1}$ is a~submartingale under $\Po$. Using first Markov's, then Jensen's, and finally Doob's maximal inequality, we obtain 
	\begin{equation*}
		\PP \left[ \max_{n< S_1} Y_n > x \right]
		\leq x^{-\gamma} \EE \left(\max_{n< S_1} Y_n\right)^\gamma
		\leq x^{-\gamma} \E \left( \Eo \max_{n < \xi_1} Y_{n}^2\right)^{\gamma/2}
		\leq x^{-\gamma} \E \left( 4 \Eo Y_{\xi_1-1}^2\right)^{\gamma/2}.
	\end{equation*}
	If additionally $\E\xi^\gamma < \infty$, then we may estimate the penultimate term using the second part of Lemma~\ref{lem:maxGW}. Since $Y_0 = 0$, we have
	\begin{equation*}
		\Eo \max_{n < \xi_1} Y_{n}^2 \leq 16 \xi_1^2,
	\end{equation*}
	thus
	\begin{equation*}
		\PP \left[ \max_{n< S_1} Y_n > x \right] \leq 16^{\gamma/2} \E\xi^\gamma x^{-\gamma} .
	\end{equation*}
	
	To estimate the second term in \eqref{eq:Ydecomp}, observe that
	\begin{equation*}
		\left(Y_{S_1 + j}\right)_{j \in \N} \od \left(\sum_{k=1}^{\Y_1} U_j^{(k)}\right)_{j\in\N},
	\end{equation*}
	where the $U^{(k)}$'s are (independent under $\Po$) copies of the process $U$, independent of $\Y_1$ under $\PP$. Conditioning on $\Y_1$ and using the second part of Lemma~\ref{lem:Nimmigrants}, we obtain
	\begin{equation*}
		\PP \left[ \max_{n \geq S_1} Y_n > x\right] \leq C_1 \EE\Y_1^{\gamma} x^{-\gamma},
	\end{equation*}
	which concludes the proof of the first part of the lemma. If $\E\xi^\gamma\rho^\gamma < \infty$, we may estimate $\EE\Y_1^\gamma$. Under $\Po$,
	\begin{equation*}
		\Y_1 \od \sum_{k=1}^{Y_{\xi_1-1} + 1} G_k,
	\end{equation*}
	where $G_k \sim Geo(\lambda_1)$ are independent of $Y_{\xi_1-1}$ and each other. Moreover, as was explained in the proof of Lemma~\ref{lem:maxGW}, $Y_{\xi_1-1} \sim Geo(1/\xi_1)$ under $\Po$. In particular, $\Eo Y_{\xi_1-1} = \xi_1-1$ and $\Eo Y_{\xi_1-1}^2 = 2\xi_1^2 - 3\xi_1+1$, therefore
	\begin{equation*}
		\Eo\Y_1^2 = \Eo\left[(Y_{\xi_1-1} + 1) (2\rho_1^2 + \rho_1) + (Y_{\xi_1-1}^2 + Y_{\xi_1-1}) \rho_1^2\right] = 2\xi_1^2\rho_1^2 + \xi_1\rho_1.
	\end{equation*}
	Jensen's inequality and subadditivity of the function $x \mapsto x^{\gamma/2}$ give
	\begin{equation*}
		\EE\Y_1^\gamma \leq \E\left(\Eo\Y_1^2\right)^{\gamma/2} \leq 2^{\gamma/2}\E\xi^\gamma\rho^\gamma + \E\xi^{\gamma/2}\rho^{\gamma/2} < \infty,
	\end{equation*}
	which proves \eqref{eq:immmaxdrift}.
	
\end{proof}

\section{Proof of Theorem \ref{thm:A}}\label{sec:thmA}

Throughout this section we work under assumptions $(A)$. In the proof of Theorem \ref{thm:A} we will use the fact that the extinctions divide the process $Z$ into independent fragments. That is, we first determine tail asymptotics of the maximum up to time $S_{\tau_1}$.

For any $A>0$ denote $\sigma(A) = \inf\{n: \Z_n \geq A\}$. The next lemma is an analogue of Lemma~4 in \cite{kesten1975limit} and can be proved the very same way, that is by examining $\Eo[\Z_k^{\alpha} | \Z_{k-1}]$ using methods we've seen in previous proofs.

\begin{lem}\label{lem:Zsigmamoment}
	For any fixed $A>0$, $0 < \EE[\Z_{\sigma(A)}^\alpha \1_{\sigma(A) < \tau_1} ] < \infty$.
\end{lem}

The main proof strategy is as follows: we choose sufficiently big $A$ and argue that neither the particles living before time $S_{\sigma(A)}$, nor the descendants of the immigrants arriving after this time contribute significantly to the examined maximum. Therefore its behaviour is determined by $\Z_{\sigma(A)}$ particles in the generation $S_{\sigma(A)}$ and their progeny.

Let us first take care of the particles alive before time $S_{\sigma(A)}$.

\begin{lem}\label{lem:beforesigma}	
	For any fixed $A$,
	\begin{equation*}
		\PP\left[\max_{n \leq S_{\sigma(A)} \wedge S_{\tau_1}} Z_n > x \right] = o(x^{-\alpha}).
	\end{equation*}
\end{lem}

\begin{proof}
	Fix $A>0$ and put $\sigma = \sigma(A)$. First, observe that
	\begin{equation*}
		x^\alpha \PP\left[\Z_{\sigma} > x, \,  \sigma < \tau_1\right] = \EE\left[ x^\alpha \1_{\Z_\sigma > x, \sigma < \tau_1} \right] \leq \EE\left[ \Z_\sigma^\alpha \1_{\Z_\sigma > x, \sigma < \tau_1} \right] \to 0
	\end{equation*}
	as $x\to \infty$ by Lemma~\ref{lem:Zsigmamoment} and the dominated convergence theorem. Since $\Z_{\tau_1} = 0$ by definition, this implies
	\begin{equation*}
		\PP\left[ Z_{S_\sigma \wedge S_{\tau_1}} > x \right] = o(x^{-\alpha}).
	\end{equation*}
	
	Next, let $x > A$. The only generations before time $S_\sigma$ in which the population size may exceed $x$ are the unmarked ones. However, since $\Z_k < A$ for $k < \sigma$, the maximum of $Z$ in generations $S_{k-1}+1, \dots S_{k}-1$ is stochastically dominated by $M_k^A$, the maximum of Galton-Watson process with $Geo(1/2)$ offspring distribution, unit immigration and $A$ initial particles, evolving for time $\xi_{k}$. Observe that
	\begin{equation*}
		\begin{split}
			\PP\left[\max_{n < S_\sigma \wedge S_{\tau_1}} Z_n > x \right]
			& \leq \PP\left[\max_{k < x^{\delta/2}} M_k^A > x \right] + \PP\left[ \tau_1 > x^{\delta/2} \right] \\
			& \leq x^{\delta/2} \PP\left[M_1^A > x \right] + \PP\left[ \tau_1 > x^{\delta/2} \right].
		\end{split}
	\end{equation*}
	Since $\alpha+\delta \leq 2$, by Markov's and Jensen's inequalities,
	\begin{equation*}
		\PP\left[ M_1^A > x \right] \leq x^{-\alpha -\delta} \E \left(\Eo (M_1^A)^2 \right)^{(\alpha+\delta)/2}.
	\end{equation*}
	The second part of Lemma~\ref{lem:maxGW} implies that
	\begin{equation*}
		\Eo (M_1^A)^2 \leq 16(\xi_1^2 + A\xi_1 + A^2)
	\end{equation*}
	and thus, since $x \mapsto x^{(\alpha+\delta)/2}$ is subadditive,
	\begin{equation*}
		x^{\delta/2} \PP\left[M_1^A > x \right] \leq x^{-\alpha - \delta/2} 16^{(\alpha+\delta)/2}\left( \E\xi^{\alpha + \delta} + A^{(\alpha+\delta)/2}\E\xi^{(\alpha + \delta)/2} + A^{\alpha+\delta}\right)
		= o(x^{-\alpha}).
	\end{equation*}
	The second term may be bounded using Lemma~\ref{lem:tau}, that is
	\begin{equation*}
		\PP\left[ \tau_1 > x^{\delta/2} \right] \leq e^{-c x^{\delta/2}} \EE e^{c\tau_1} = o(x^{-\alpha}),
	\end{equation*}
	which ends the proof.
	
\end{proof}

The next lemma assures that the contribution of progeny of immigrants arriving after $S_{\sigma(A)}$ is negligible. Recall that $Y(k)$ counts the progeny of immigrants arriving in the $k$'th block, that is in generations $S_{k-1}, S_{k-1}+1, \dots S_k -1$.

\begin{lem}\label{lem:aftersigma}
	Fix $\eps > 0$. There exists $A_1(\eps)$ such that for $A > A_1(\eps)$,
	\begin{equation}\label{eq:maxaftersigma}
		\PP\left[ 2 \sum_{k=\sigma(A)+1}^{\tau_1} \max_{n \geq 1} Y_n(k) > \eps x \right] \leq \eps x^{-\alpha}.
	\end{equation}
\end{lem}
\begin{proof}
	For any $A>0$, with $\sigma=\sigma(A)$,
	\begin{equation*}
		\begin{split}
			\PP\left[ 2 \sum_{k=\sigma+1}^{\tau_1} \max_{n \geq 1} Y_n(k) > \eps x \right]
			& = \PP \left[ \sum_{k=1}^\infty \1_{\sigma \leq k < \tau_1} \max_{n \geq 1} Y_n(k+1) > \eps x/2 \right] \\
			& \leq \sum_{k=1}^\infty \PP \left[\sigma \leq k < \tau_1, \max_{n \geq 1} Y_n(k+1) > \eps x / 4k^2 \right].
		\end{split}
	\end{equation*}
	Observe that the event $\{\sigma \leq k < \tau_1\}$ is defined in terms of $Z_1, \dots Z_{S_k}$, while the process $Y(k+1)$ evolves in the environment given by $(\xi_j, \rho_j)$ for $j \geq k+1$, hence is independent of $Z_1, \dots Z_{S_k}$. Moreover, the second part of Lemma~\ref{lem:immmax} applied with $\gamma=\alpha$ gives tail bounds on the maximum of $Y(k+1)$. That is,
	\begin{multline*}
		\sum_{k=1}^\infty \PP \left[ \sigma \leq k < \tau_1,  \max_{n \geq 1} Y_n(k+1) > \eps x / 4k^2 \right] \\
		\begin{aligned}
			& = \sum_{k=1}^\infty \PP \left[\sigma \leq k < \tau_1\right] \PP \left[\max_{n \geq 1} Y_n(k+1) > \eps x / 4k^2 \right] \\
			& \leq C_3 \sum_{k=1}^\infty \PP \left[\sigma \leq k < \tau_1\right] (\eps x / 4k^2)^{-\alpha} \\
			& = C_3 4^{\alpha} (\eps x)^{-\alpha} \EE\left[ \sum_{k=\sigma}^{\tau_1-1} k^{2\alpha}\1_{\sigma < \tau_1}\right] \\
			& \leq C_3 4^{\alpha} \eps^{-\alpha} x^{-\alpha} \EE\left[ \tau_1^{2\alpha + 1} \1_{\sigma < \tau_1} \right].
		\end{aligned}
	\end{multline*}
	Since $\EE\tau_1^{2\alpha + 1} < \infty$ and $\sigma(A) \topr \infty$ as $A \to \infty$, $\EE[\tau_1^{2\alpha+1} \1_{\sigma(A)<\tau_1}] \to 0$ as $A \to \infty$ by the dominated convergence theorem. One may thus find $A_1(\eps)$ such that \eqref{eq:maxaftersigma} holds for $A>A_1(\eps)$.
	
\end{proof}

We gave bounds on the generations sizes of particles alive before time $S_{\sigma(A)}$ and those coming from immigrants arriving after that time. What is left is investigating behaviour of the particles residing exactly in generation $S_{\sigma(A)}$ and their progeny. First, we show that the maximal generation size among these is controlled by the initial number of particles and the maxima of the potential $\Psi$ defined in \eqref{def:psi}. Then we show that for any $A>0$, the latter has regularly varying tails with index $-\alpha$.

For $k\geq S_\sigma$, where $\sigma = \sigma(A)$ for fixed $A>0$, let $V_{\sigma, k}$ be the number of progeny of the particles from generation $S_\sigma$ residing in generation $k$ and let $\V_{\sigma,n} = V_{\sigma,S_n}$; in particular, $\Z_\sigma = \V_{\sigma,\sigma}$. Recall the variables $\Psi_{m,k}$ defined in \eqref{def:psi}.

\begin{lem}\label{lem:Zsigmaprog}
	For any $\eps > 0$ there exists $A_2(\eps)$ such that for $A > A_2(\eps)$,
	\begin{equation*}
		\PP \left[ \left|\max_{k\geq S_\sigma} (V_{\sigma,k}+V_{\sigma,k+1}) - \Z_\sigma\max_{k \geq S_\sigma} (\Psi_{\sigma+1,k} + \Psi_{\sigma+1,k+1})\right| > \eps x, \sigma < \tau_1 \right]
		\leq \eps x^{-\alpha} \EE\left[\Z_\sigma^{\alpha} \1_{\sigma < \tau_1}\right],
	\end{equation*}
	where $\sigma = \sigma(A)$.
\end{lem}

\begin{proof}
	We begin by estimating the difference of maxima within one block. Observe that the potential $\Psi$ is constant within each block, therefore for any $n \in \N$,
	\begin{equation*}
		\begin{split}
			& \left| \max_{S_n \leq k < S_{n+1}} (V_{\sigma,k} + V_{\sigma,k+1})
			- \Z_\sigma \max_{S_n \leq k < S_{n+1}}  (\Psi_{\sigma+1,k} + \Psi_{\sigma+1,k+1}) \right| \\
			& \leq
			\left| \max_{S_n \leq k < S_{n+1}-1} (V_{\sigma,k} + V_{\sigma,k+1})
			- 2\Z_\sigma \Pi_{\sigma+1,n} \right| \\
			& + |V_{\sigma,S_{n+1}-1} + V_{\sigma, S_{n+1}} - \Z_\sigma\Pi_{\sigma+1,n} - \Z_\sigma\Pi_{\sigma+1, n+1} |
		\end{split}
	\end{equation*}
	Let us estimate the first term. Since
	\begin{equation*}
		\max_{S_n \leq k < S_{n+1}-1} (V_{\sigma,k}+V_{\sigma,k+1}) = 2\V_{\sigma,n} + \max_{S_n \leq k < S_{n+1}-1} \left(V_{\sigma,k} + V_{\sigma,k+1} - 2\V_{\sigma,n}\right),
	\end{equation*}
	we have
	\begin{equation*}
		\begin{split}
			\left| \max_{S_n \leq k < S_{n+1}-1} (V_{\sigma,k} + V_{\sigma,k+1})
			- 2 \Z_\sigma \Pi_{\sigma+1,n} \right|
			& \leq 2\left(\left|\V_{\sigma,n} - \Z_\sigma\Pi_{\sigma+1,n}\right|
			+ \max_{S_n \leq k < S_{n+1}} |V_{\sigma,k} - \V_{\sigma,n}|\right).
		\end{split}
	\end{equation*}
	The second term may be estimated simply by
	\begin{equation*}
		\begin{split}
			&|V_{\sigma,S_{n+1}-1} + V_{\sigma, S_{n+1}} - \Z_\sigma\Pi_{\sigma+1,n} - \Z_\sigma\Pi_{\sigma+1, n+1} | \\
			& \leq |\V_{\sigma, n+1} - \Z_\sigma \Pi_{\sigma+1,n+1}|
			+  |\V_{\sigma,n} - \Z_\sigma \Pi_{\sigma+1,n} |
			+ |V_{\sigma, S_{n+1}-1} - \V_{\sigma,n}|,
		\end{split}
	\end{equation*}
	which gives
	\begin{equation*}
		\begin{split}
			& \left| \max_{S_n \leq k < S_{n+1}} (V_{\sigma,k} + V_{\sigma,k+1})
			- \Z_\sigma \max_{S_n \leq k < S_{n+1}}  (\Psi_{\sigma+1,k} + \Psi_{\sigma+1,k+1}) \right| \\
			& \leq 3 |\V_{\sigma,n}-\Z_\sigma\Pi_{\sigma+1,n}| + 3 \max_{S_n \leq k < S_{n+1}} |V_{\sigma,k} - \V_{\sigma,n}| + |\V_{\sigma,n+1}-\Z_\sigma\Pi_{\sigma+1,n+1}|.
		\end{split}
	\end{equation*}
	Next, in view of
	\begin{equation*}
		\begin{split}
			&\left|\max_{k\geq S_\sigma} (V_{\sigma,k}+V_{\sigma,k+1}) - \Z_\sigma \max_{k \geq S_\sigma} (\Psi_{\sigma+1,k} + \Psi_{\sigma+1,k+1}) \right| \\
			& = \left| \max_{n\geq \sigma} \max_{S_n \leq k < S_{n+1}} (V_{\sigma,k} + V_{\sigma,k+1})
			- \max_{n \geq \sigma } \Z_\sigma \max_{S_n \leq k < S_{n+1}}  (\Psi_{\sigma+1,k} + \Psi_{\sigma+1,k+1}) \right| \\
			& \leq \sum_{n \geq \sigma} \left| \max_{S_n \leq k < S_{n+1}} (V_{\sigma,k} + V_{\sigma,k+1})
			- \Z_\sigma \max_{S_n \leq k < S_{n+1}}  (\Psi_{\sigma+1,k} + \Psi_{\sigma+1,k+1}) \right|,
		\end{split}
	\end{equation*}
	the above estimations give
	\begin{equation*}
		\begin{split}
			\PP & \left[ \left|\max_{k\geq S_\sigma} (V_{\sigma,k}+V_{\sigma,k+1}) - \Z_\sigma\max_{k \geq S_\sigma} (\Psi_{\sigma+1,k} + \Psi_{\sigma+1,k+1})\right| > \eps x, \sigma < \tau_1 \right] \\
			& \leq \PP \left[ 4\sum_{n\geq \sigma} \left| \V_{\sigma,n} - \Z_\sigma\Pi_{\sigma+1,n}\right| > \eps x/2, \sigma < \tau_1 \right] \\
			& + \PP \left[ 3\sum_{n\geq \sigma} \max_{S_n \leq k < S_{n+1}} |V_{\sigma,k} - \V_{\sigma,n}| > \eps x/2, \sigma < \tau_1 \right].
		\end{split}
	\end{equation*}
	Both terms can be estimated by Lemma~\ref{lem:Nimmigrants} applied with $\gamma = \alpha$. Conditioned on $(\sigma, Z_1, \dots Z_{S_\sigma})$, the process $(V_{\sigma,n})_{n\geq S_{\sigma}}$ is a~sum of $\Z_\sigma$ independent copies of the process $U$. We have, on the event $\{\sigma < \tau_1\}$,
	\begin{equation*}
		\begin{split}
			\PP&\left[4\sum_{n\geq \sigma} \left| \V_{\sigma,n} - \Z_\sigma\Pi_{\sigma+1,n}\right| > \eps x/2 \,\Bigg| \, \sigma, Z_1, \dots Z_{S_\sigma} \right] \leq C_1 (\eps x/8)^{-\alpha} \Z_\sigma^{\alpha/2},
		\end{split}
	\end{equation*}
	which gives
	\begin{equation*}
		\begin{split}
			\PP\left[ 4\sum_{n\geq \sigma} \left| \V_{\sigma,n} - \Z_\sigma\Pi_{\sigma+1,n}\right| > \eps x/2, \sigma < \tau_1 \right]
			\leq C_1 8^\alpha (\eps x)^{-\alpha} \EE\left[ \Z_\sigma^{\alpha/2} \1_{\sigma < \tau_1}\right].
		\end{split}
	\end{equation*}
	Similarly,
	\begin{equation*}
		\PP \left[ 3 \sum_{n\geq \sigma} \max_{S_n < k < S_{n+1}} |V_{\sigma,k} - \V_{\sigma,n}| > \eps x/2, \sigma < \tau_1 \right]
		\leq C_1 6^\alpha (\eps x)^{-\alpha} \EE\left[ \Z_\sigma^{\alpha/2} \1_{\sigma < \tau_1} \right].
	\end{equation*}
	Therefore, for some constant $C'$,
	\begin{multline*}
		\PP \left[ \left|\max_{k\geq S_\sigma} (V_{\sigma,k}+V_{\sigma,k+1}) - \Z_\sigma\max_{k \geq S_\sigma} (\Psi_{\sigma+1,k} + \Psi_{\sigma+1,k+1})\right| > \eps x, \sigma < \tau_1 \right] \\
		\leq C' (\eps x)^{-\alpha} \EE\left[\Z_\sigma^{\alpha/2} \1_{\sigma < \tau_1}\right].
	\end{multline*}
	Finally, for any fixed $\eps > 0$, since $\Z_\sigma \geq A$, we have
	\begin{equation*}
		\EE\left[ \Z_\sigma^{\alpha/2} \1_{\sigma < \tau_1} \right]
		\leq A^{-\alpha/2} \EE\left[ \Z_\sigma^{\alpha} \1_{\sigma < \tau_1} \right]
	\end{equation*}
	and one may choose $A_2(\eps)$ large enough for the claim to hold.
	
\end{proof}

\begin{lem}\label{lem:maxPsi}
	There exists $c_\Psi \in (0,\infty)$ such that for any fixed $A>0$,
	\begin{equation}
		\PP\left[ \Z_\sigma \max_{k \geq S_\sigma} (\Psi_{\sigma+1,k} + \Psi_{\sigma+1,k+1}) > x, \sigma < \tau_1 \right] \sim c_\Psi \EE\left[ \Z_\sigma^\alpha \1_{\sigma < \tau_1} \right] x^{-\alpha},
	\end{equation}
	where $\sigma = \sigma(A)$.
\end{lem}

\begin{proof}
	Since the sequence $\Psi_{\sigma+1,k}$ is constant on the blocks between marked points, we have
	\begin{equation*}
		\max_{k \geq S_\sigma} (\Psi_{\sigma+1,k} + \Psi_{\sigma+1,k+1}) = \max_{n \geq \sigma} \left(2\1_{\xi_{n+1} > 1} \vee (1+\rho_{n+1})\right)\Pi_{\sigma+1,n}.
	\end{equation*}
	Observe that
	\begin{equation*}
		\log\left(\left(2\1_{\xi_{n+1} > 1} \vee (1+\rho_{n+1})\right)\Pi_{1,n}\right) = \sum_{k=1}^n \log(\rho_k) + \log(2\1_{\xi_{n+1} > 1} \vee (1+\rho_{n+1}))
	\end{equation*}
	is a~perturbed random walk\footnote{A process $(X_n)_{n\in\N}$ is called a~\emph{perturbed random walk} if it is of the form $X_n = \sum_{k=1}^{n-1}A_k + B_n$ for i.i.d.\ pairs of random variables $((A_n,B_n))_{n\in\N}$. We refer the reader to \cite{iksanow2016renewal} for a~detailed study of perturbed random walks.}. By Theorem 1.3.8 in \cite{iksanow2016renewal}, assumptions $(A)$ guarantee that
	\begin{equation*}
		\PP \left[ \max_{n \geq 0} (2\1_{\xi_{n+1} > 1} \vee (1+\rho_{n+1})) \Pi_{1,n} > x \right] \sim c_\Psi x^{-\alpha}
	\end{equation*}
	for a~constant $c_\Psi \in (0,\infty)$ given by
	\begin{equation*}
		c_\Psi = \E(2^\alpha \1_{\xi_1>1} \vee (1+\rho_1)^\alpha - \max_{n\geq 2} (2^\alpha \1_{\xi_{n+1} > 1} \vee (1+\rho_{n+1})^\alpha)\Pi_{1,n}^\alpha)_+.
	\end{equation*}
	
	Note that the variables $\Z_\sigma \1_{\sigma < \tau_1}$ and $\max_{n \geq \sigma} (2\1_{\xi_{n+1} > 1} \vee (1+\rho_{n+1})) \Pi_{\sigma+1, n}$ are independent under $\PP$. Therefore, by Breiman's lemma \cite[Lemma~B.5.1]{buraczewski:2016:power},
	\begin{multline*}
		\PP \left[ \Z_\sigma \max_{k \geq S_\sigma} (\Psi_{\sigma+1,k} + \Psi_{\sigma+1,k+1}) > x, \sigma < \tau_1 \right] \\
		= \PP \left[ \Z_\sigma \1_{\sigma < \tau_1} \cdot \max_{n \geq \sigma} (2\1_{\xi_{n+1} > 1} \vee (1+\rho_{n+1})) \Pi_{\sigma+1,n} > x \right]
		\sim \EE\left[ \Z_\sigma^\alpha \1_{\sigma < \tau_1} \right] c_\Psi x^{-\alpha}. 
	\end{multline*}
\end{proof}

The rest of the proof is standard. First, all the lemmas proven so far allow us to determine the asymptotics of the maximum in time $[0,S_{\tau_1})$. Then we use the fact that the extinctions divide our process into independent pieces.

\begin{prop}\label{prop:maxinblock}
	For some constant $c_M>0$,
	\begin{equation*}
		\PP \left[ \max_{0 \leq n < S_{\tau_1}} (Z_n + Z_{n+1}) > x \right] \sim c_M x^{-\alpha}.
	\end{equation*}
\end{prop}
\begin{proof}
	Fix $\eps > 0$ and take $A > A(\eps) := \max\{A_1(\eps), A_2(\eps)\}$, $\sigma = \sigma(A)$. First, observe that
	\begin{multline*}
		\PP \left[ \max_{S_\sigma \leq n < S_{\tau_1}} (Z_n+Z_{n+1}) > x, \sigma < \tau_1 \right]
		\leq \PP \left[ \max_{0 \leq n < S_{\tau_1}} (Z_n+Z_{n+1}) > x \right] \\
		\leq \PP \left[ \max_{S_\sigma \leq n < S_{\tau_1}} (Z_n + Z_{n+1}) > x, \sigma < \tau_1 \right]
		+ \PP \left[ \max_{n < S_\sigma \wedge S_{\tau_1}} (Z_n+Z_{n+1}) > x\right].
	\end{multline*}
	Lemma~\ref{lem:beforesigma} ensures that for large enough $x$,
	\begin{equation*}
		\PP \left[ \max_{n < S_\sigma \wedge S_{\tau_1}} (Z_n+Z_{n+1}) > x\right] \leq \PP \left[ 2\max_{n \leq S_\sigma \wedge S_{\tau_1}} Z_n > x\right] \leq \eps x^{-\alpha}.
	\end{equation*}
	Recall that by $Y(k) = (Y_j(k))_{j\in\Z}$ we denoted the process counting the progeny of immigrants arriving in the $k$'th block, with the convention $Y_j(k) = 0$ for $j < 0$. For $n \geq S_\sigma$,
	\begin{equation*}
		Z_n = V_{\sigma,n} + \sum_{k=\sigma+1}^{\tau_1} Y_{n-S_{k-1}}(k),
	\end{equation*}
	thus
	\begin{multline*}
		\PP\left[ \max_{S_\sigma \leq n < S_{\tau_1}} (V_{\sigma,n} + V_{\sigma, n+1}) > x, \sigma < \tau_1 \right]
		\leq \PP \left[ \max_{S_\sigma \leq n < S_{\tau_1}} (Z_n + Z_{n+1}) > x, \sigma < \tau_1 \right]\\
		\leq \PP\left[ \max_{S_\sigma \leq n < S_{\tau_1}} (V_{\sigma,n} + V_{\sigma,n+1}) > (1-\eps)x, \sigma < \tau_1 \right]
		+ \PP\left[ 2 \sum_{k=\sigma+1}^{\tau_1} \max_{n \geq 1} Y_n(k) > \eps x \right]
	\end{multline*}
	and Lemma~\ref{lem:aftersigma} ensures that
	\begin{equation*}
		\PP\left[ 2 \sum_{k=\sigma+1}^{\tau_1} \max_{n \geq 1} Y_n(k) > \eps x \right] \leq \eps x^{-\alpha}.
	\end{equation*}
	Finally,
	\begin{equation*}
		\begin{split}
			\PP&\left[ \Z_\sigma \max_{k \geq S_\sigma} (\Psi_{\sigma,k} + \Psi_{\sigma,k+1}) > (1+\eps)x, \sigma < \tau_1 \right] \\
			& - \PP \left[ \left|\max_{k\geq S_\sigma} (V_{\sigma,k}+V_{\sigma,k+1}) - \Z_\sigma\max_{k \geq S_\sigma} (\Psi_{\sigma,k} + \Psi_{\sigma,k+1})\right| > \eps x, \sigma < \tau_1 \right] \\
			& \leq \PP\left[ \max_{S_\sigma \leq n < S_{\tau_1}} (V_{\sigma,n}+V_{\sigma,n+1}) > x, \sigma < \tau_1 \right] \\
			& \leq \PP\left[ \Z_\sigma \max_{k \geq S_\sigma} (\Psi_{\sigma,k} + \Psi_{\sigma,k+1}) > (1-\eps)x, \sigma < \tau_1 \right] \\
			& + \PP \left[ \left|\max_{k\geq S_\sigma} (V_{\sigma,k}+V_{\sigma,k+1}) - \Z_\sigma\max_{k \geq S_\sigma} (\Psi_{\sigma,k} + \Psi_{\sigma,k+1})\right| > \eps x, \sigma < \tau_1 \right],
		\end{split}
	\end{equation*}
	and by Lemma~\ref{lem:Zsigmaprog},
	\begin{equation*}
		\PP \left[ \left|\max_{k\geq S_\sigma} (V_{\sigma,k}+V_{\sigma,k+1}) - \Z_\sigma\max_{k \geq S_\sigma} (\Psi_{\sigma,k} + \Psi_{\sigma,k+1})\right| > \eps x, \sigma < \tau_1 \right] \leq \eps x^{-\alpha} \EE\left[ \Z_\sigma^\alpha \1_{\sigma < \tau} \right].
	\end{equation*}
	
	Putting things together and invoking Lemma~\ref{lem:maxPsi} and Lemma~\ref{lem:Zsigmamoment} we get that for any $\eps > 0$ such that $\eps(1+\eps)^{\alpha} < c_\Psi$ and for any $A > A(\eps)$,
	\begin{multline*}
		0 < ((1+\eps)^{-\alpha}c_\Psi - \eps) \EE\left[ \Z_\sigma^\alpha \1_{\sigma < \tau_1} \right] \\
		\leq \liminf_{x\to\infty} x^\alpha \PP\left[ \max_{0 \leq n < S_{\tau_1}} (Z_n+Z_{n+1}) > x \right]
		\leq \limsup_{x\to\infty} x^\alpha \PP\left[ \max_{0 \leq n < S_{\tau_1}} (Z_n+Z_{n+1}) > x \right] \\
		\leq ((1-2\eps)^{-\alpha}c_\Psi + \eps) \EE\left[ \Z_\sigma^\alpha \1_{\sigma < \tau_1} \right] + 2\eps < \infty.
	\end{multline*}
	Observe that this relation implies that both the limits 
	$$\lim_{x\to\infty} x^\alpha \PP\left[ \max_{0 \leq n < S_{\tau_1}} (Z_n+Z_{n+1}) > x \right] \quad \textnormal{and}\quad  \lim_{A\to\infty}\EE\left[ \Z_{\sigma(A)}^\alpha \1_{\sigma(A) < \tau_1} \right]$$
	exist, are positive and satisfy
	\begin{equation*}
		\lim_{x\to\infty} x^\alpha \PP\left[ \max_{0 \leq n < S_{\tau_1}} (Z_n+Z_{n+1}) > x \right] = c_\Psi \lim_{A\to\infty} \EE\left[ \Z_{\sigma(A)}^\alpha \1_{\sigma(A) < \tau_1} \right] =: c_M.
	\end{equation*}
\end{proof}

Due to Lemma~\ref{lem:Sn-to-n} and the relation \eqref{eq:Zk-to-Lk}, the next result implies Theorem \ref{thm:A}.

\begin{thm}\label{thm:AZ}
	Under assumptions $(A)$,
	\begin{equation*}
		\PP\left[ n^{-1/\alpha} \max_{0 \leq k < S_n} (Z_k + Z_{k+1}) > x\right] \xrightarrow{n \to \infty} 1 - \exp\left(-\frac{c_M}{\EE\tau_1} x^{-\alpha}\right)
	\end{equation*}
	for every $x > 0$.
\end{thm}

\begin{proof}
	Since the extinctions divide the process $Z$ into independent fragments, an immediate corollary of Proposition \ref{prop:maxinblock} is that
	\begin{equation*}
		\PP\left[ n^{-1/\alpha} \max_{0 \leq k < S_{\tau_n}} (Z_k + Z_{k+1}) > x\right] \xrightarrow{n \to \infty} 1 - \exp(-c_M x^{-\alpha})
	\end{equation*}
	(c.f.\ \cite[Proposition 1.11]{resnick:2013:extreme}). Lemma~\ref{lem:tau} implies that $\EE\tau_1 < \infty$. Therefore passing from the maximum up to time $S_{\tau_n}$ to the maximum up to $S_n$ may be done exactly as in the proof of Lemma~\ref{lem:Sn-to-n}.
\end{proof}

\section{Proof of Theorem \ref{thm:B}}\label{sec:thmB}

As we saw in the proof of Theorem \ref{thm:A}, the limiting behaviour of maxima in case $(A)$ comes from the tail asymptotics of the variable $M_\Psi$ defined in \eqref{def:Mpsi}. The assumption $\E\xi^{\alpha+\delta} < \infty$ implies that for every $k$, $\max_{j<\xi_k} Y_j(k)$ is negligible. In terms of the random walk, this means that the time the walker spends in a~block when crossing it for the first time is negligible. As we will see, under assumptions $(B)$ it is not; the maximal local time is obtained when the walker crosses a~particularly long block for the first time, by their visits to sites within this block and potentially excursions to the left.

Consider a~simple symmetric random walk on $\Z$ and denote by $\bar{L}_k(n)$ the number of times the walk visits site $k$ before reaching $n$. Consider $(\bar{L}_s(n))_{s\in[0,n]}$ being a~piecewise linear interpolation of $(\bar{L}_{k}(n) )_{0 \leq k\leq n}$. The Ray-Knight theorem \cite{ray:1963, knight:1963} states that
$$ \left(\frac{1}{n} \bar{L}_{n(1-t)}(n)\right)_{t \in [0,1]} \tod \left(B_t\right)_{t \in [0,1]}$$
in $C[0,1]$ as $n\to\infty$, where $B$ is a~squared Bessel process which may be defined as
\begin{equation}\label{def:Bessel}
	B_t = \|W(t)\|^2,
\end{equation}
for $W(t) = (W_1(t),W_2(t))$ being a~standard two-dimensional Brownian motion with $W(0)= 0$. By the continuous mapping theorem,
\begin{equation}\label{eq:maxL}
	\left( \frac{1}{n} \max_{k \leq n} \bar{L}_k(n), \frac{1}{n}\bar{L}_0(n) \right) \tod (M_B, B(1)),
\end{equation}
where $M_B = \sup\{B_t: t \in [0,1]\}$. 

With this at hand, we may inspect the maximal local time that the RWSRE obtains when crossing a~(long) block between marked points for the first time. To this end, consider a~walk starting at $0$ in the environment that has marked points only on the non-positive half-line, and stop it when it reaches point $N$. By the Ray-Knight theorem, the limit of maximal local time in the interval $[1,N]$, where the walk is symmetric, scaled by $N$, is $M_B$. As we saw in the proof of Theorem \ref{thm:A}, the number of visits in the negative half-line should be controlled by the number of visits to $1$ and the maxima of the potential $\Psi$.

In the associated branching process, the steps of the walk during its first crossing of a~block between marked points are counted by the process $Y$. Therefore our goal is to understand the growth of the maximal generation in the process $Y$ as the size of the first block -- in which the immigrants arrive -- tends to infinity. To this end, for any $N\in\N$ let $Y^{(N)}$ be a~BPSRE evolving in an environment with fixed $\xi_1 = N$ and such that the immigrants arrive only in generations up to $(N-1)$'th.

\begin{lem}\label{lem:oneblockmax}
	Under assumptions $(B)$,
	\begin{equation}
		\frac{1}{N}\max_{k\geq 0} \left(Y^{(N)}_k + Y^{(N)}_{k+1}\right) \tod M_\infty \quad \textnormal{as } N \to \infty
	\end{equation}
	for $M_\infty = \max(M_B, B(1)M_\Psi/2)$, where $M_\Psi$ is a~copy of the variable defined in \eqref{def:Mpsi} independent of the Bessel process $B$.
\end{lem}
\begin{proof}
	To simplify the notation we will write $Y$ instead of $Y^{(N)}$. Observe that \eqref{eq:maxL} and the duality between branching process and random walk imply
	$$ \left( \frac{1}{N}\max_{k \leq N-2}(Y_k + Y_{k+1}), \frac{1}{N}(Y_{N-1} + Y_{N-2}) \right) \tod (M_B, B(1)). $$
	However, since the particles in generation $N-1$ are children of those from $(N-2)$'th and an immigrant, born with distribution $Geo(1/2)$, we have
	\begin{equation*}
		\EE\left(Y_{N-1} - Y_{N-2} - 1\right)^2 = \EE (Y_{N-1} - \EE\left[Y_{N-1} \, | \, Y_{N-2} \right])^2 = 2(\EE Y_{N-2} + 1) = 2(N-1),
	\end{equation*}
	which, together with Chebyshev's inequality, implies that $(Y_{N-1}-Y_{N-2})/N \topr 0$ and thus
	\begin{equation*}
		\left( \frac{1}{N}\max_{k \leq N-2}(Y_k + Y_{k+1}), \frac{Y_{N-1}}{N} \right) \tod (M_B, B(1)/2).
	\end{equation*}
	Moreover, the variables $Y_k$ for $k \leq N-1$ are independent of the environment, in particular of $\Psi_{1,n}, n\geq 0$.
	
	From here on we proceed as in the proof of Lemma~\ref{lem:Zsigmaprog}, to show that the maximum in generations after $(N-1)$'th is comparable with $Y_{N-1}M_\Psi$. That is, we use Lemma~\ref{lem:Nimmigrants} applied with $\gamma = \beta$ to obtain, for some constant $C>0$,
	\begin{equation}\label{eq:YafterN}
		\PP \left[ \left|\max_{k\geq N} (Y_{k}+Y_{k+1}) - \Y_1\max_{k \geq N} (\Psi_{2,k} + \Psi_{2,k+1})\right| >  x \right]
		\leq C x^{-\beta} \EE \Y_1^{\beta/2}
	\end{equation}
	for any $x > 0$. The particles in the first marked generation $S_1 = N$ are born with distribution $Geo(\lambda_1)$ from those counted by $Y_{{N-1}}$ and an immigrant. Therefore we have $\Eo \Y_1 = N\rho_1$, and by Jensen's inequality,
	\begin{equation*}
		\EE \Y_1^{\beta/2} \leq N^{\beta/2} \E\rho^{\beta/2}.
	\end{equation*}
	Moreover, we may calculate quenched moments of $\Y_1$ conditioned on $Y_{N-1}$ to get an analogue of \eqref{Uevolution}. We obtain
	\begin{equation}\label{eq:Y1toYN}
		\begin{split}
			\EE\left|\Y_1 - \rho_1 Y_{N-1} \right|^{\beta}
			& \leq \E\left(\Eo(\Y_1 - \rho_1 Y_{N-1} )^2\right)^{\beta/2} \\
			& = \E\left((\Eo Y_{N-1}(\rho_1^2 + \rho_1) + 2\rho_1^2 + \rho_1 \right)^{\beta/2} \\
			& \leq (N^{\beta/2}+1) (2^{\beta/2}\E\rho^{\beta} + \E\rho^{\beta/2}),
		\end{split}
	\end{equation}
	where the last inequality follows from subadditivity of $x \mapsto x^{\beta/2}$ and the fact that $\Eo Y_{N-1} = N-1$. Observe that $\max_{k\geq N}(\Psi_{2,k} + \Psi_{2,k+1}) \leq 2 + M_{\Psi,2}$ and by \eqref{eq:psimoment}, $\E M_\Psi^\beta < \infty$. Therefore, since $(Y_{N-1}, \Y_1, \rho_1)$ is independent of $(\rho_j)_{j\geq 2}$, we have
	\begin{multline}\label{eq:YatYN}
		\PP\left[ \left| \Y_1 \max_{k\geq N}(\Psi_{2,k} + \Psi_{2,k+1}) - \rho_1 Y_{N-1}\max_{k\geq N}(\Psi_{2,k} + \Psi_{2,k+1}) \right| > x \right] \\
		\leq x^{-\beta} \E (2+M_\Psi)^{\beta} \EE|\Y_1 - \rho_1Y_{N-1}|^{\beta} \leq C' x^{-\beta} (N^{\beta/2}+1)
	\end{multline}
	for some constant $C'>0$ and any $x > 0$.
	
	Observe that \eqref{eq:YafterN} and \eqref{eq:YatYN} imply that for any fixed $\eps > 0$,
	\begin{equation*}
		\begin{split}
			\PP &\left[ \left|\max_{k\geq N} (Y_{k}+Y_{k+1}) - Y_{N-1}\max_{k \geq N} (\Psi_{1,k} + \Psi_{1,k+1})\right| > \eps N \right] \\
			& \leq \PP \left[ \left|\max_{k\geq N} (Y_{k}+Y_{k+1}) - \Y_1\max_{k \geq N} (\Psi_{2,k} + \Psi_{2,k+1})\right| >  \eps N/2 \right] \\
			& + \PP\left[ \left| \Y_1 \max_{k\geq N}(\Psi_{2,k} + \Psi_{2,k+1}) - \rho_1 Y_{N-1}\max_{k\geq N}(\Psi_{2,k} + \Psi_{2,k+1}) \right| > \eps N/2 \right] \\
			& \leq (\eps N/2)^{-\beta} \left(C N^{\beta/2} \E\rho^{\beta/2} + C'(N^{\beta/2} + 1)\right) = O(N^{-\beta/2}).
		\end{split}
	\end{equation*}
	Finally, by \eqref{eq:Y1toYN}, for any $\eps>0$,
	\begin{equation*}
		\PP\left[ |\Y_1 - \rho_1 Y_{N-1}| > \eps N \right] \leq \eps^{-\beta} (N^{-\beta/2} + N^{-\beta}) (2^{\beta/2}\E\rho^\beta + \E\rho^{\beta/2}) = O(N^{-\beta/2}),
	\end{equation*}
	therefore the weak limit of
	\begin{equation*}
		\frac{1}{N}\max_{k\geq 0} (Y_k + Y_{k+1}) = \frac{1}{N} \max\left( \max_{k\leq N-2} (Y_k + Y_{k+1}), Y_{N-1} + \Y_1, \max_{k \geq N} (Y_k + Y_{k+1})\right)
	\end{equation*}
	is the same as that of
	\begin{multline*}
		\frac{1}{N} \max\left( \max_{k\leq N-2} (Y_k + Y_{k+1}), Y_{N-1}(1 + \rho_1), Y_{N-1}\max_{k\geq N}(\Psi_{1,k} + \Psi_{1,k-1})\right) \\
		= \frac{1}{N} \max\left( \max_{k\leq N-2} (Y_k + Y_{k+1}), Y_{N-1}M_{\Psi,1}\right)
	\end{multline*}
	which is $\max(M_B, B(1)M_\Psi/2)$ by the continuous mapping theorem.
\end{proof}

\begin{rem}\label{rem:Mmoment}
	Under assumptions $(B)$, $\EE M_\infty^{\beta + \delta} < \infty$. Indeed, by \eqref{def:Bessel},
	\begin{equation*}
		M_B^2 = \sup\left\{\left(W_1(t)^2 + W_2(t)^2\right)^2 \, : \, t \in [0,1]\right\},
	\end{equation*}
	where $W_1, W_2$ are independent one-dimensional Brownian motions. Doob's maximal inequality applied to $W_1, W_2$ implies that $\EE M_B^2 <\infty$. Since $\beta+\delta \leq 2$, it follows that $\EE M_B^{\beta+\delta} < \infty$. Moreover, by \eqref{eq:psimoment}, $\EE M_{\Psi}^{\beta+\delta} < \infty$, and since $M_\Psi$ and $B$ are independent, we have
	\begin{equation*}
		\EE M_\infty^{\beta+\delta} \leq \EE M_B^{\beta+\delta} \EE(1 + M_\Psi/2)^{\beta+\delta} < \infty.
	\end{equation*}
\end{rem}

Recall that the process $Y(k)$ counts the progeny of immigrants arriving in the $k$'th block. Since Lemma~\ref{lem:oneblockmax} suggests that the maximum of process $Y(k)$ should be comparable with $\xi_k M_\infty$ when $\xi_k$ is large, we begin the proof of Theorem \ref{thm:B} by distinguishing large blocks in the environment. Recall the sequence $(a_n)_{n\in\N}$ defined in \eqref{def:an}. Fix $\eps > 0$ and let
\begin{equation*}
	I_{n,\eps} = \{ k \leq n \, : \, \xi_k > \eps a_n\}, \quad I_{n,\eps}^c = \{ k \leq n \, : \, \xi_k \leq \eps a_n\}.
\end{equation*}
For fixed $n$ and $k \leq n$, we will say that the $k$'th block is large if $k \in I_{n,\eps}$, and small otherwise.

It follows from the definition of the sequence $(a_n)_{n\in\N}$ and regular variation of the tails of $\xi$ that for any $x > 0$,
\begin{equation}\label{eq:anconv}
	n \P[\xi > x a_n] \to x^{-\beta}, \quad n \to \infty.
\end{equation}
Therefore, by Proposition 3.21 in \cite{resnick:2013:extreme},
\begin{equation}\label{eq:PPP}
	\sum_{k=1}^n \delta_{(\xi_k/a_n, k/n)} \tod P_\mu,
\end{equation}
where $P_\mu$ is a~Poisson point process on $(0,\infty]\times[0,\infty)$ with intensity measure $d\mu(x,t) = \beta x^{-\beta-1} dx dt$. In particular, as $n\to \infty$, the sequence of variables $|I_{n,\eps}|$, which count the number of large blocks, converges weakly to a~Poisson distribution with parameter $\eps^{-\beta}$. 

We begin by showing that all the progeny of immigrants arriving in small blocks is negligible.
\begin{prop}\label{prop:smalltr}
	There is a~constant $C_5$ such that for any $\eps > 0$ and $\bar{\eps} > 0$,
	\begin{equation*}
		\limsup_{n\to\infty} \PP\left[ \max_{j \geq 1} \sum_{k \in I_{n,\eps}^c}  Y_{j-S_{k-1}}(k) > \bar{\eps} a_n \right] \leq C_5 \bar{\eps}^{-\beta-\delta} \eps^{\delta}.
	\end{equation*}
\end{prop}

\begin{proof}
	Let
	\begin{equation*}
		\eta_n = \inf\{k > 0 : \tau_k > n\}.
	\end{equation*}
	Since $\EE\tau_1 < \infty$ by Lemma~\ref{lem:tau}, the strong law of large numbers implies $\eta_n/n \to \eta := 1/\EE\tau$ as $n \to \infty$, $\PP$-a.s. We have
	\begin{equation*}
		\begin{split}
			\PP\left[ \max_{j \geq 1} \sum_{k \in I_{n,\eps}^c}  Y_{j-S_{k-1}}(k) > \bar{\eps} a_n \right]
			&\leq \PP\left[ \max_{j \geq 1} \sum_{k \leq \tau_{2n\eta}}  Y_{j-S_{k-1}}(k) \1_{\xi_k \leq \eps a_n} > \bar{\eps} a_n \right] \\
			&  + \PP\left[ |\eta - \eta_n/n| > \eta \right].
		\end{split}
	\end{equation*}
	The second term tends to $0$ as $n\to\infty$. Since the extinctions divide our process into i.i.d.\ pieces, we have
	\begin{equation*}
		\begin{split}
			\PP\left[ \max_{j \geq 1} \sum_{k \leq \tau_{2n\eta}}  Y_{j-S_{k-1}}(k) \1_{\xi_k \leq \eps a_n} > \bar{\eps} a_n \right] 
			& \leq \sum_{m = 1}^{2n\eta} \PP\left[ \max_{j \geq 1} \sum_{k=\tau_{m-1}}^{\tau_{m}} Y_{j-S_{k-1}}(k) \1_{\xi_k \leq \eps a_n} > \bar{\eps} a_n \right] \\
			= 2n\eta & \, \PP\left[ \max_{j \geq 1} \sum_{k=0}^{\tau_{1}} Y_{j-S_{k-1}}(k) \1_{\xi_k \leq \eps a_n} > \bar{\eps} a_n \right] \\
			\leq 2n\eta & \, \PP \left[ \sum_{k = 0}^{\tau_1} \max_{j \geq 1} Y_j(k) \1_{\xi_k \leq \eps a_n} > \bar{\eps} a_n \right] \\
			= 2n\eta & \, \PP\left[ \sum_{k = 1}^\infty \1_{k \leq \tau_1} \max_{j \geq 1} Y_j(k) \1_{\xi_k \leq \eps a_n} > \bar{\eps} a_n \right] \\
			\leq 2n\eta & \, \sum_{k = 1}^\infty \PP\left[ \tau_1 \geq k\right] \PP \left[ \max_{j \geq 1} Y_j(k) \1_{\xi_k \leq \eps a_n} > \bar{\eps} a_n/2k^2 \right],
		\end{split}
	\end{equation*}
	where in the last line we used the fact that $\{\tau_1 \geq k\}$ and the process $Y(k)$ are independent.
	
	Since the environment is given by an i.i.d.\ sequence, it is enough to estimate the tails of the maximum of the process $(Y_j \1_{\xi_1 \leq \eps a_n})_{j\in\N}$. By Lemma~\ref{lem:immmax} applied with $\gamma = \beta + \delta$,
	\begin{equation*}
		\PP\left[ \max_{j \geq 1} Y_j\1_{\xi_1 \leq \eps a_n} > x \right] \leq C_2 x^{-\gamma} \left( \E\left(\Eo Y_{\xi_1-1}^2 \1_{\xi_1 \leq \eps a_n}\right)^{\gamma/2} + \EE \Y_1^\gamma\1_{\xi_1 \leq \eps a_n}\right).
	\end{equation*}
	As we calculated in the proof of Lemma~\ref{lem:immmax},
	\begin{equation*}
		\Eo Y_{\xi_1-1}^2\1_{\xi_1 \leq \eps a_n} = \xi_1(\xi_1-1)\1_{\xi_1 \leq \eps a_n}, \quad \Eo\Y_1^2\1_{\xi_1 \leq \eps a_n} = (2\xi_1^2\rho_1^2 + \xi_1\rho_1)\1_{\xi_1 \leq \eps a_n},
	\end{equation*}
	therefore
	\begin{equation*}
		\E\left(\Eo Y_{\xi_1-1}^2 \1_{\xi_1 \leq \eps a_n}\right)^{\gamma/2} \leq \E\xi^\gamma\1_{\xi\leq \eps a_n}
	\end{equation*}
	and
	\begin{equation*}
		\EE \Y_1^\gamma \1_{\xi_1 \leq \eps a_n} \leq \E \left(\Eo \Y_1^2 \1_{\xi_1 \leq \eps a_n}\right)^{\gamma/2} \leq \left( 2^{\gamma/2} \E\rho^\gamma + \E\rho^{\gamma/2} \right) \E\xi^\gamma \1_{\xi \leq \eps a_n}.
	\end{equation*}
	Putting things together, for some constant $C>0$ and any $x>0$,
	\begin{equation*}
		\PP\left[ \max_{j \geq 1} Y_j\1_{\xi_1 \leq \eps a_n} > x \right] \leq C x^{-\gamma} \E\xi^\gamma \1_{\xi \leq \eps a_n} \leq C x^{-\gamma} \int_0^{\eps a_n} t^{\gamma - 1} \P[\xi > t] dt.
	\end{equation*}
	By Karamata's theorem \cite[Theorem 1.5.11]{bingham:1987:regular} and \eqref{eq:anconv}, 
	\begin{equation*}
		\int_0^{\eps a_n} t^{\gamma - 1} \P[\xi > t] dt \sim \frac{1}{\gamma+\beta} (\eps a_n)^{\gamma} \P[\xi > \eps a_n] \sim \frac{1}{\gamma+\beta} \eps^{\gamma-\beta} a_n^\gamma n^{-1}.
	\end{equation*}
	Using those estimates, we obtain, for some constants $C',C''>0$,
	\begin{equation*}
		\begin{split}
			\PP\left[ \max_{j \geq 1} \sum_{k \leq \tau_{2n\eta}}  Y_{j-S_{k-1}}(k) \1_{\{\xi_k \leq \eps a_n\}} > \bar{\eps} a_n \right] 
			& \leq C' n \sum_{k=1}^\infty \PP[\tau_1 \geq k] \left(\bar{\eps} a_n/2k^2\right)^{-\gamma} \eps^{\gamma-\beta} a_n^\gamma n^{-1} \\
			& \leq C'' \bar{\eps}^{-\gamma} \eps^{\gamma - \beta} \EE\tau_1^{2\gamma+1},
		\end{split}
	\end{equation*}
	which finishes the proof since $\gamma = \beta + \delta$ and $\EE\tau_1^{2\gamma+1} < \infty$ by Lemma~\ref{lem:tau}.
\end{proof}

The next step is to investigate the maximal generations among the progeny of immigrants from large blocks. Although it may happen that the descendants of particles from several large blocks coexist in one generation of the process $Z$, we will show later that it is unlikely, so that we may begin by investigating the maxima of $|I_{n,\eps}|$ {\it independent} processes, each representing the progeny of immigrants from a~large block. To this end, assume that our probability space contains variables $\left\{\left(Y^{(N)}_k(j)\right)_{k\in \N} \, : \, j,N \in \N \right\}$ such that
\begin{itemize}
	\setlength\itemsep{0em}
	\item the processes $\left(Y^{(N)}_k(j)\right)_{k\in \N}$ are i.i.d.\ copies of $\left(Y^{(N)}_k\right)_{k\in\N}$,
	\item the family $\left\{\left(Y^{(N)}_k(j)\right)_{k\in \N} : j,N\in\N \right\}$ is independent of the environment $\{(\xi_k,\lambda_k)\}_{k\in\Z}$.
\end{itemize}
For any $j,N \in \N$ denote 
\begin{equation}\label{def:Mcoup}
	M_N(j) = \max_{k \geq 0} \left(Y^{(N)}_k(j) + Y^{(N)}_{k+1}(j)\right).
\end{equation}

\begin{prop}\label{prop:largetr}
	Fix $\eps > 0$. For any $x > 0$,
	\begin{equation*}
		\lim_{n\to\infty} \PP\left[ \max_{j \in I_{n,\eps}} M_{\xi_j}(j) > x a_n \right] = 1 - \exp\left(-x^{-\beta}\EE M_\infty^\beta \1_{M_\infty < x/\eps} - \eps^{-\beta}\PP[M_\infty \geq x/\eps]\right).
	\end{equation*}
\end{prop}
\begin{proof}
	
	Recall that $|I_{n,\eps}|$ converges in distribution to $Pois(\eps^{-\beta})$. Moreover, conditioning on $|I_{n,\eps}| = k$, the examined maximum is a~maximum of $k$ independent variables with distribution given by
	\begin{equation*}
		\PP\left[ M_{\xi} \in \cdot \, \big| \, \xi > \eps a_n \right]
	\end{equation*}
	for $\xi$ independent of $\{Y^{(N)}, M_N \, : \, N\in\N \}$. In particular,
	\begin{equation}\label{eq:maxIn}
		\PP\left[ \max_{j \in I_{n,\eps}} M_{\xi_j}(j) > x a_n \right]
		= 1 - \E\left[(1-\PP\left[M_\xi > x a_n \, | \, \xi > \eps a_n\right])^{|I_{n,\eps}|}\right].
	\end{equation}
	
	Fix $\bar{\eps} > 0$. By Lemma~\ref{lem:oneblockmax}, $M_N/N \tod M_\infty$ as $N\to\infty$. Observe that the distribution of $(M_B, B(1))$ is absolutely continuous. Since $M_\infty = \max(M_B, B(1)M_\Psi/2)$ for $M_\Psi$ independent of the Bessel process $B$, it follows from an application of Fubini's and the dominated convergence theorems that the distribution of $M_\infty$ is also absolutely continuous. In particular, the cumulative distribution functions of $M_N/N$ converge uniformly to the cumulative distribution function of $M_\infty$. That is, there exists $N_0 \in \N$ such that for $N > N_0$,
	\begin{equation*}
		\sup_{y>0}|\PP[M_N/N > y] - \PP[M_\infty > y]| < \bar{\eps}.
	\end{equation*}
	In particular, for $\xi$ independent of $M_\infty$ and $n$ such that $\eps a_n > N_0$,
	\begin{equation}\label{eq:est2}
		\begin{split}
			\sup_{y > 0}& \left| \PP\left[ M_{\xi}>  y \, | \, \xi > \eps a_n \right] - \PP\left[M_\infty > y/\xi \, | \, \xi > \eps a_n \right] \right| \\
			&= \sup_{y>0} \frac1{\P[\xi > \eps a_n]}\left| \sum_{N>\eps a_n} \left(\PP[M_N/N > y/N] - \PP[M_\infty > y/N]\right)\P[\xi=N] \right| \\
			& \leq \frac1{\P[\xi > \eps a_n]} \sum_{N > \eps a_n} \bar{\eps} \P[\xi=N] \\
			&= \bar{\eps}.
		\end{split}
	\end{equation}
	Observe that
	\begin{equation*}
		\begin{split}
			\PP&\left[ M_\infty > x a_n/\xi \, | \, \xi > \eps a_n \right]
			= \frac{\PP[\xi M_\infty > x a_n, \xi > \eps a_n]}{\P[\xi > \eps a_n]} \\
			& = \frac{1}{\P[\xi > \eps a_n]} \left(\int_{[0,x/\eps)}\P[\xi > x a_n/t] \PP[M_\infty \in dt] + \int_{[x/\eps, \infty)} \P[\xi > \eps a_n] \PP[M_\infty \in dt] \right) \\
			& = \int_{[0,x/\eps)}\frac{\P[\xi > x a_n/t]}{\P[\xi > \eps a_n]} \PP[M_\infty \in dt] + \PP[M_\infty \geq x/\eps].
		\end{split}
	\end{equation*}
	By the uniform convergence theorem for regularly varying functions (see (B.1.2) in \cite{buraczewski:2016:power}), for $n$ large enough,
	\begin{equation*}
		\sup_{c \geq 1} \left|\frac{\P[\xi > c \eps a_n]}{\P[\xi > \eps a_n]} - c^{-\beta} \right| < \bar{\eps},
	\end{equation*}
	which means that
	\begin{multline*}
		\left| \int_{[0,x/\eps)}\frac{\P[\xi > x a_n/t]}{\P[\xi > \eps a_n]} \PP[M_\infty \in dt] - x^{-\beta}\eps^\beta \EE M_\infty^\beta \1_{M_\infty < x/\eps} \right| \\ =
		\left| \int_{[0,x/\eps)}\frac{\P[\xi > x a_n/t]}{\P[\xi > \eps a_n]} - \left(\frac{x}{t\eps}\right)^{-\beta} \PP[M_\infty \in dt] \right| < \bar{\eps}.
	\end{multline*}
	Therefore,
	\begin{equation*}
		\left|\PP[M_\infty > x a_n/\xi \, | \, \xi > \eps a_n] - \left(x^{-\beta} \eps^\beta \EE M_\infty^\beta \1_{M_\infty < x/\eps} + \PP[M_\infty \geq x/\eps]\right)\right| < \bar{\eps},
	\end{equation*}
	which together with \eqref{eq:est2} implies, for large $n$,
	\begin{equation*}\label{eq:est3}
		\left| \PP\left[ M_{\xi}>  x a_n \, | \, \xi > \eps a_n \right] - \left(x^{-\beta} \eps^\beta \EE M_\infty^\beta \1_{M_\infty < x/\eps} + \PP[M_\infty \geq x/\eps]\right) \right| < 2\bar{\eps}.
	\end{equation*}
	Putting this estimate to \eqref{eq:maxIn} and using the fact that $|I_{n,\eps}| \tod Pois(\eps^{-\beta})$, we obtain
	\begin{multline*}
		1 - \exp\left(-\eps^{-\beta}\left(x^{-\beta} \eps^\beta \EE M_\infty^\beta \1_{M_\infty < x/\eps} + \PP[M_\infty \geq x/\eps] -2\bar{\eps}\right)\right) \\
		\leq \liminf_{n\to\infty} \PP\left[ \max_{k \leq n} M_{\xi_k}(k)\1_{\xi_k > \eps a_n} > x a_n \right] \leq \limsup_{n\to\infty} \PP\left[ \max_{k \leq n} M_{\xi_k}(k)\1_{\xi_k > \eps a_n} > x a_n \right] \\
		\leq 1 - \exp\left(-\eps^{-\beta}\left(x^{-\beta} \eps^\beta \EE M_\infty^\beta \1_{M_\infty < x/\eps} + \PP[M_\infty \geq x/\eps] +2\bar{\eps}\right)\right),
	\end{multline*}
	which finishes the proof since $\bar{\eps}$ is arbitrary.
\end{proof}

We are now ready to prove Theorem \ref{thm:B}, rephrased into the setting of the associated branching process.

\begin{thm}
	Under assumptions $(B)$,
	\begin{equation*}
		\PP\left[ a_n^{-1} \max_{0 \leq k < S_n} (Z_k + Z_{k+1}) > x\right] \xrightarrow{n \to \infty} 1 - \exp\left(-\EE M_\infty^\beta x^{-\beta}\right)
	\end{equation*}
	for every $x>0$.
\end{thm}
\begin{proof}
	Fix $\eps > 0$. For any $\bar{\eps} > 0$,
	\begin{multline}\label{eq:trsplit}
		\PP\left[ \max_{j \geq 1} \sum_{k \in I_{n,\eps}} (Y_{j-S_{k-1}}(k) + Y_{j-S_{k-1} + 1}(k)) > x a_n \right]
		\leq \PP\left[\max_{j < S_n} (Z_j + Z_{j+1}) > x a_n \right] \\
		\leq \PP\left[2 \max_{j \geq 1} \sum_{k\in I_{n,\eps}^c} Y_{j-S_{k-1}}(k) > \bar{\eps} a_n \right] \\
		+ \PP\left[ \max_{j \geq 1} \sum_{k \in I_{n,\eps}} (Y_{j-S_{k-1}}(k) + Y_{j-S_{k-1} + 1}(k)) > (x-\bar{\eps})a_n \right].
	\end{multline}
	
	Note that because of \eqref{eq:PPP} we expect that for large $n$ the set $I_{n,\eps}$ should be distributed rather uniformly on $\{1,\dots n\}$, so that the large blocks are far from each other. Indeed, since $n\P[\xi > \eps a_n] \to \eps^{-\beta}$, for any sequence $b_n$ such that $b_n = o(n)$,
	\begin{equation*}
		\P\left[ \left(\exists k,l \in I_{n,\eps} \right) \, k\neq l, |k-l|\leq b_n \right] \leq n\P[\xi > \eps a_n] \cdot b_n \P[\xi > \eps a_n] \to 0 \quad \textnormal{ as } n\to\infty.
	\end{equation*}
	That is, with high probability, large blocks are at distance at least $b_n$ from each other. On the other hand, we know that extinction occurs very often in our process, which should mean that as the process evolves, no two bloodlines of immigrants from large blocks coexist at one time. Let
	\begin{equation*}
		D_{k,n} = \left\{ \Y_{\sqrt{n}}(k) = 0 \right\}
	\end{equation*}
	be an event that the progeny of immigrants from the $k$'th block does not survive more than $\sqrt{n}$ blocks. Then, by Lemma~\ref{lem:tau},
	\begin{equation}\label{eq:Ydeathest}
		\PP\left[ \bigcup_{k\leq n} D_{k,n}^c \right] 
		\leq n \PP[\tau_1 > \sqrt{n}] \leq n e^{-c\sqrt{n}} \EE e^{c\tau_1} \to 0
	\end{equation}
	as $n \to 0$. Therefore the probability of the set
	\begin{equation*}
		D_n = \{ (\forall k\in I_{n,\eps}) \,\Y_{\sqrt{n}}(k) = 0 \}
	\end{equation*}
	converges to $1$ as $n \to \infty$ and so does the probability of
	\begin{equation*}
		A_n = \{ \left( \forall k,l \in I_{n,\eps} \right) k\neq l \implies |k-l| > 2\sqrt{n}\}.
	\end{equation*}
	Moreover, on the set $A_n \cap D_n$, the progeny of immigrants from each large block dies out before the next large block occurs. That is, $\max_{j \geq 1} \sum_{k \in I_{n,\eps}} (Y_{j-S_{k-1}}(k) + Y_{j-S_{k-1} + 1}(k))$ is a~maximum of independent maxima of $Y(k)$ such that $k \in I_{n,\eps}$. In particular,
	\begin{multline}\label{eq:Ytoiid}
		\PP\left[ \max_{j \geq 1} 	\sum_{k \in I_{n,\eps}} (Y_{j-S_k}(k) + Y_{j-S_k + 1}(k)) > x a_n, \, A_n\cap D_n \right] \\
		= \PP\left[ \max_{k \in I_{n,\eps}} M_{\xi_k}(k) > x a_n , \, A_n \cap \bar{D}_{n}\right]
	\end{multline}
	for the variables $M_{\xi_k}(k)$ defined in \eqref{def:Mcoup} and
	\begin{equation*}
		\bar{D}_n = \left\{ (\forall k \in I_{n,\eps}) \, \Y_{\sqrt{n}}^{(\xi_k)}(k) = 0 \right\}.
	\end{equation*}
	Note that an estimate identical to \eqref{eq:Ydeathest} implies that $\PP[\bar{D}_n]\to 1$ as $n\to\infty$. Therefore, \eqref{eq:Ytoiid} and Proposition \ref{prop:largetr} give, for any $x > 0$,
	\begin{multline*}
		\lim_{n\to\infty} \PP\left[ \max_{j \geq 1} 	\sum_{k \in I_{n,\eps}} (Y_{j-S_k}(k) + Y_{j-S_k + 1}(k)) > x a_n\right]
		= \lim_{n\to\infty}\PP\left[ \max_{k \in I_{n,\eps}} M_{\xi_k}(k) > x a_n\right]\\
		= 1 - \exp\left(-x^{-\beta}\EE M_\infty^\beta \1_{M_\infty < x/\eps} - \eps^{-\beta}\PP[M_\infty \geq x/\eps]\right).
	\end{multline*}
	
	Going back to \eqref{eq:trsplit}, we have
	\begin{multline}\label{eq:liminf}
		1 - \exp\left(-x^{-\beta}\EE M_\infty^\beta \1_{M_\infty < x/\eps} - \eps^{-\beta}\PP[M_\infty \geq x/\eps]\right) \\ \leq \liminf_{n\to\infty} \PP\left[\max_{j < S_n} (Z_j + Z_{j+1}) > x a_n \right].
	\end{multline}
	On the other hand, by Proposition \ref{prop:smalltr},
	\begin{equation*}
		\limsup_{n\to\infty} \PP\left[2 \max_{j \geq 1} \sum_{k\in I_{n,\eps}^c} Y_{j-S_{k-1}}(k) > \bar{\eps} a_n \right] \leq C_5 (\bar{\eps}/2)^{-\beta-\delta} \eps^\delta,
	\end{equation*}
	which means that
	\begin{equation}\label{eq:limsup}
		\begin{split}
			\limsup_{n\to\infty} \PP&\left[\max_{j < S_n} (Z_j + Z_{j+1}) > x a_n \right]
			\leq C_5 (\bar{\eps}/2)^{-\beta-\delta} \eps^\delta \\
			& + 1 - \exp\left(-(x-\bar{\eps})^{-\beta}\EE M_\infty^\beta \1_{M_\infty < (x-\bar{\eps})/\eps} - \eps^{-\beta}\PP[M_\infty \geq (x-\bar{\eps})/\eps]\right).
		\end{split}
	\end{equation}
	Observe that, since $\EE M_\infty^{\beta + \delta} < \infty$ (see Remark \ref{rem:Mmoment}), we have
	\begin{equation*}
		\eps^{-\beta}\PP[M_\infty \geq x/\eps] \leq \eps^\delta x^{-\beta-\delta} \EE M_\infty^{\beta+\delta} \to 0 \quad \textnormal{ as } \eps \to 0,
	\end{equation*}
	while by the monotone convergence theorem,
	\begin{equation*}
		\EE M_\infty^{\beta} \1_{M_\infty < x/\eps} \to \EE M_\infty^\beta \quad \textnormal{ as } \eps \to 0.
	\end{equation*}
	Therefore taking $\eps$ to $0$ in \eqref{eq:liminf} gives
	\begin{equation*}
		1 - \exp\left(-x^{-\beta}\EE M_\infty^\beta \right) \leq \liminf_{n\to\infty} \PP\left[\max_{j < S_n} (Z_j + Z_{j+1}) > x a_n \right],
	\end{equation*}
	and similarly in \eqref{eq:limsup},
	\begin{equation*}
		\limsup_{n\to\infty} \PP\left[\max_{j < S_n} (Z_j + Z_{j+1}) > x a_n \right]
		\leq 1 - \exp\left(-(x-\bar{\eps})^{-\beta}\EE M_\infty^\beta\right),
	\end{equation*}
	which ends the proof since $\bar{\eps}>0$ is arbitrary.
\end{proof}

\section*{Acknowledgements}
The research was supported by the National Science Center, Poland (Opus, grant number 2020/39/B/ST1/00209).


\end{document}